\documentclass[prb,aps, twocolumn, nofootinbib, superscriptaddress, longbibliography,  floatfix]{revtex4-1}

\pdfoutput=1
\usepackage[utf8]{inputenc}
\DeclareUnicodeCharacter{00A0}{ }
\usepackage{amsmath}
\usepackage{amssymb}
\usepackage{amsthm}
\usepackage{bbold}
\usepackage{mathtools}
\usepackage{physics}
\usepackage[makeroom]{cancel}
\usepackage{stmaryrd}

\usepackage{graphicx}
\usepackage{color}
\usepackage[english]{babel}
\usepackage{stmaryrd}
\usepackage{microtype}
\usepackage{float}

\usepackage[percent]{overpic}
\usepackage{tikz}
\usetikzlibrary{calc,shadows,matrix,positioning,patterns,decorations.pathreplacing}

\usepackage{hyperref}
\hypersetup{
	pdftoolbar=true,
	pdfmenubar=true,
	pdffitwindow=false,
	pdfstartview={FitH},
	pdftitle={},
	pdfauthor={Kaye, Chen, Parcollet},
	pdfsubject={},
	pdfcreator={},
	pdfproducer={},
	pdfnewwindow=true,
	colorlinks=true,
	linkcolor=black,
	citecolor=blue,
	filecolor=magenta,
	urlcolor=blue
}

\usepackage{xr}

\newcommand{\RR}{\mathbb{R}}
\newcommand{\CC}{\mathbb{C}}

\newcommand{\OO}[1]{\mathcal{O}\paren{#1}}
\newcommand{\paren}[1]{\left( #1 \right)}

\newcommand{\pd}[2]{\frac{\partial#1}{\partial#2}}
\newcommand{\wh}[1]{\widehat{#1}}
\newcommand{\wt}[1]{\widetilde{#1}}
\newcommand{\wb}[1]{\overline{#1}}

\renewcommand{\Re}{\operatorname{Re}}
\renewcommand{\Im}{\operatorname{Im}}
\newcommand{\bea}{\begin{eqnarray}}
\newcommand{\eea}{\end{eqnarray}}
\newcommand{\be}{\begin{equation}}
\newcommand{\ee}{\end{equation}}

\newcommand{\epsm}{\epsilon_{\mbox{\tiny mach}}}

\newcommand{\omax}{\omega_{\max}}
\newcommand{\rhodlr}{\rho_{\text{DLR}}}
\newcommand{\gdlr}{G_{\text{DLR}}}
\newcommand{\hdlr}{H_{\text{DLR}}}
\newcommand{\gb}{G_{\text{B}}}
\newcommand{\kb}{K_{\text{B}}}
\newcommand{\rhob}{\rho_{\text{B}}}
\newcommand{\errmat}{E}
\newcommand{\errfun}{e}
\newcommand{\nmax}{n_{\max}}
\newcommand{\kmat}{\mathcal{K}}

\newtheorem{theorem}{Theorem}
\newtheorem{lemma}{Lemma}

\makeatletter
\newenvironment{proofw}{\par
  \pushQED{\qed}%
  \normalfont \topsep6\p@\@plus6\p@\relax
  \trivlist
  \item[]\ignorespaces
}{%
  \popQED\endtrivlist\@endpefalse
}
\makeatother

\newcommand{\CCQ}{Center for Computational Quantum Physics, Flatiron Institute, 162 5th Avenue, New York, NY 10010, USA}
\newcommand{\CCM}{Center for Computational Mathematics, Flatiron Institute, 162 5th Avenue, New York, NY 10010, USA}

\begin{document}

\title{Discrete Lehmann representation of imaginary time Green's functions}

\author{Jason Kaye}
\email{jkaye@flatironinstitute.org}
\affiliation{\CCM}
\affiliation{\CCQ}

\author{Kun Chen}
\affiliation{\CCQ}

\author{Olivier Parcollet}
\affiliation{\CCQ}
\affiliation{Universit\'e Paris-Saclay, CNRS, CEA, Institut de Physique Th\'eorique, 91191, Gif-sur-Yvette, France}

\begin{abstract}
   We present an efficient basis for imaginary time Green's functions
   based on a low rank decomposition of the spectral Lehmann
   representation. The basis functions are simply a set of well-chosen
   exponentials, so the corresponding expansion may be thought of as a
   discrete form of the Lehmann 
   representation using an effective spectral density which is
   a sum of $\delta$ functions. The basis is determined only by an upper bound on the product $\beta \omax$, with $\beta$ the inverse
   temperature and $\omax$ an energy cutoff, and a user-defined error tolerance
   $\epsilon$. The number $r$ of basis functions
   scales as $\OO{\log(\beta \omax)
   \log \paren{1/\epsilon}}$. The discrete Lehmann representation of a particular
   imaginary time Green's function can be
   recovered by interpolation at a set of $r$ imaginary time nodes. Both the basis functions and
   the interpolation nodes can be
   obtained rapidly using standard numerical linear algebra routines.
   Due to the simple form of the basis, the discrete Lehmann
   representation of a Green's function can be explicitly transformed to the Matsubara
   frequency domain, or obtained directly by interpolation on a
   Matsubara frequency grid.
   We benchmark the efficiency of the representation on simple cases,
   and with a high precision 
   solution of the Sachdev-Ye-Kitaev equation at low temperature. 
   We compare our approach with the related intermediate
   representation method, and introduce an improved algorithm to build
   the intermediate representation basis and a corresponding sampling
   grid.
\end{abstract}

\maketitle

\section{Introduction}

Quantum many-body physics is entering a new era, with the rise of
high precision algorithms capable of obtaining
controlled solutions in the strongly interacting regime.
A large family of approaches concentrates on computing finite temperature correlation functions.
Indeed, the imaginary time formalism in thermal equilibrium is well suited 
to describe both the thermodynamic and many of the equilibrium
properties of a system.\cite{Abrikosov:QFT}
It is widely used, for example by quantum Monte Carlo algorithms, which are formulated in imaginary time.

For many applications, generic methods of representing one and two-particle imaginary time Green's
functions may be insufficient to obtain the required precision given
computational cost and memory
constraints. Examples include
(i) the storage of one-body Green's functions with a large number of
orbitals, as in quantum chemistry applications (see Ref.
\onlinecite{GullStrand_2020}
and the references therein), or on a lattice with complex momentum dependence;
(ii) the high precision solution of the Dyson equation for such Green's
functions; 
(iii) computations in which highly accurate representations of
Green's functions are required, as
for the bare propagator in some high-order perturbative expansions
\cite{qqmc2020}; and
(iv) the storage of two-body Green's functions, which depend on three
time arguments \cite{shinaoka18,shinaoka20}.

The simplest approach is to represent a Green's function $G$ on
a uniform grid of $m$ points in imaginary time $\tau$, and by a truncated
Fourier series of $m$ modes in
Matsubara frequency $i \nu_n$. While this method offers some practical
advantages, including the ability to transform between the imaginary
time and Matsubara frequency domains by means of the fast Fourier
transform, it is a poor choice from the
point of view of efficiency, particularly when the inverse temperature
$\beta$ is large.
First, $m = \OO{\beta}$ grid points are required in imaginary time to resolve sharp
features caused by high energy scales. Second, since the
Green's functions are discontinuous at the endpoints $\tau = 0$ and
$\beta$ of the imaginary time interval, their Fourier
coefficients decay as $\OO{1/m}$, so that the representation converges
with low-order accuracy.

Representing $G(\tau)$ by an orthogonal polynomial (Chebyshev or
Legendre) expansion of degree $m$ yields a significant
improvement.\cite{Boehnke2011}
Indeed, since $G(\tau)$ is smooth on $[0,\beta]$, such a representation
converges with spectral accuracy
\cite{Boehnke2011,kananenka16,gull18,gull18,GullStrand_2020}; see also
Ref. \onlinecite{trefethen19} for a thorough overview of the theory of
orthogonal polynomial approximation. However, resolving the Green's
function still requires an expansion of degree $m = \OO{\sqrt{\beta}}$.
\cite{chikano18}

A third idea is the ``power grid'' method, which uses a grid exponentially clustered towards
$\tau = 0$ and $\beta$. In this approach, an adaptive sequence
of panels is constructed, and a polynomial interpolant used on each
panel, leading to a representation requiring only
$\OO{\log \beta}$ degrees of freedom. \cite{ku00,ku02,kananenka16}
However, the power grid method has been
implemented using uniform grid interpolation on each panel, which can
lead to numerical instability for high-order interpolants. A more stable
method, using spectral grids on each panel, is incorporated as an
intermediate step in our framework, but ultimately further compression of
the representation can be achieved. 

A newer approach is to construct highly compact representations
by taking advantage of the specific structure of imaginary time Green's
functions, which satisfy the spectral Lehmann representation
\begin{equation} \label{eq:lehmann}
  G(\tau) = - \int_{-\infty}^{\infty} K(\tau,\omega)
  \rho(\omega) \, d\omega,
\end{equation}
for $\tau \in [0,\beta]$.
Here $\rho$ is the spectral density, $\omega$
is a real frequency variable,
and the kernel $K$ is given in the fermionic case by
\begin{equation} \label{eq:defK}
  K(\tau,\omega) =  \frac{e^{-\omega \tau}}{1+e^{-\beta \omega}}.
\end{equation} 
In our discussion, we assume that the support of $\rho$ is contained in
$[-\omax,\omax]$, for $\omax$ a high energy cutoff; this always holds to high accuracy for sufficiently
large $\omax$. For convenience, we also define a dimensionless high energy cutoff $\Lambda
\equiv \beta \omax$.

The key observation is that the fermionic kernel $K$ can be approximated to
high accuracy by a low rank decomposition. \cite{shinaoka17}
The most well-known manifestation of this fact is the severe
ill-conditioning of
analytic continuation from the imaginary to the real time axis.
However, one can take
advantage of this low rank structure to obtain a compact
representation of $G(\tau)$. In Refs. \onlinecite{shinaoka17,chikano18}, orthogonal bases for imaginary time
Green's functions containing only $\OO{\log \Lambda}$ basis functions are constructed from the left singular vectors in the
singular value decomposition (SVD) of
a discretization of $K$. 
This method, called the intermediate representation (IR), has been used
successfully in a variety of applications, including those
involving two-particle quantities
\cite{shinaoka18,shinaoka19,otsuki20,shinaoka20,wallerberger20,wang20,shinaoka21};
see also Ref. \onlinecite{shinaoka21_2} for a useful review and further
references. A related approach is the minimax isometry method, which
uses similar ideas to construct optimal quadrature rules for
Matsubara summation in GW applications. \cite{kaltak20}

In this paper, we present a method which is related to the IR,
but uses a different low rank
decomposition of $K$, called the interpolative
decomposition (ID). \cite{cheng05,liberty07}
It leads to a discrete Lehmann representation (DLR) of any imaginary
time Green's function $G(\tau)$ as a  
linear combination of $r$ exponentials $e^{- \omega_k \tau}$ with a
set of frequencies $\omega_{k}$ which depend only on $\Lambda$ and
$\epsilon$. Like the IR basis, the DLR basis is \emph{universal} in the sense that given any
$\Lambda$ and $\epsilon$, it is sufficient to represent any imaginary
time Green's function obeying the energy cutoff $\Lambda$ to within
accuracy $\epsilon$. The number of basis functions is observed to scale as $r = \OO{\log(\Lambda) \log
\paren{1/\epsilon}}$, and is nearly the same
as the number of IR basis functions with the same choice of $\Lambda$
and error tolerance $\epsilon$.

Our construction begins with a discretization of $K$ on a composite
Chebyshev grid, designed to resolve the range of energy
scales present in Green's functions up to a given cutoff $\Lambda$. Then,
instead of applying the SVD to the resulting
matrix as in the IR method, we use the ID to select a set of $r$ representative frequencies
$\omega_k$ such that the functions $K(\tau,\omega_k)$ form the basis of
exponentials. The ID also yields a set of $r$ 
interpolation nodes, such that the DLR of a given Green's
function $G$ can be recovered from samples at those nodes.
The DLR can be explicitly transformed to the Matsubara
frequency domain, where it takes the form of a linear combination
of $r$ poles $(i \nu_n + \omega_k)^{-1}$. As in the imaginary time
domain, the DLR can also be recovered by interpolation at $r$
nodes on the Matsubara frequency axis.

Compared with the IR approach, the DLR basis exchanges orthogonality
for a simple, explicit form of the basis functions. However, we show that
orthogonality is not required for numerically stable recovery of the
representation. On the other hand, using an explicit basis of
exponentials has many advantages. In particular, it avoids the cost and
complexity of working with the IR basis functions, which are themselves
represented on a fine adaptive grid, and evaluated using corresponding
interpolation procedures. Many standard computational tasks -- such as
transforming between the imaginary time and Matsubara frequency domains,
and performing convolutions -- are reduced to simple explicit formulas.

The algorithms which we use in the context of the DLR also carry
over to the IR method, and offer two main improvements over 
previously established algorithms.

First, the numerical tools we describe can be used to construct an
efficient sampling grid for the IR basis in a more systematic manner than the
sparse sampling method, which is typically used. Sparse sampling
provides a method of obtaining compact grids in imaginary time and Matsubara
frequency, from which one can recover the IR coefficients. \cite{li20} The sparse
sampling grid is analogous to the interpolation grid used for the DLR.
However, whereas the sparse sampling method selects a grid based on
a heuristic, we use a purely linear algebraic method with robust
accuracy guarantees, which is also applicable to the IR.

Second, existing methods to build the IR basis functions are 
computationally intensive, requiring hours of computation time for large values
of $\Lambda$. Furthermore, the basis functions themselves are
represented using a somewhat complicated adaptive data structure. Of course, basis functions for a given choice of $\Lambda$
need only be computed once and stored, and to facilitate the process, an
open source software package has been released which contains tabulated basis
functions for several fixed values of $\Lambda$, as well as routines
to work with them. \cite{chikano19} However, in some cases, the
situation is cumbersome, for example if one wishes to converge a
calculation with respect to $\Lambda$.
By contrast, we present a simple
discretization of $K(\tau,\omega)$, which allows us to construct either
the DLR or IR basis functions from a single call to the pivoted QR and
SVD algorithms, respectively, with matrices of modest size.
This yields the basis functions and associated imaginary time
interpolation nodes in less than a second on a laptop
for $\Lambda$ as large as $10^6$ and $\epsilon$ near the double machine
precision. The resulting DLR basis functions are characterized by a list of
$r$ frequency nodes $\omega_k$, and the IR basis functions are
represented using a simple data structure.

In addition to describing efficient algorithms to implement the DLR, we
present mathematical theorems  which provide error bounds and control
inequalities.  We illustrate the DLR approach on several simple
examples, as well as on a high precision, low temperature solution of
the Sachdev-Ye-Kitaev (SYK) model. \cite{SachdevYe93,gu20}

Open source Fortran and Python implementations of the
DLR are available in the library \texttt{libdlr}.\cite{libdlr} We refer the reader to
Ref. \onlinecite{kaye21} for a detailed description.

This paper is structured as follows. In Section \ref{sec:summaryintro}, we
present a short overview of the DLR
with an example, leaving aside technical details.
In Section \ref{sec:tools}, we introduce the mathematical tools required in the rest of the
paper, namely composite Chebyshev interpolation and the interpolative decomposition.
In Section \ref{sec:dlr}, we develop the DLR, describe our algorithm, and show some benchmarks.
In Section \ref{sec:ir}, we derive the IR, describe its
relationship with the DLR, and present efficient algorithms to construct
the IR basis functions and associated grid. We show how to solve the
Dyson equation efficiently using the DLR in
Section \ref{sec:dyson}, and demonstrate the method by solving the SYK
equation in Section \ref{sec:syk}.
Section \ref{sec:Conclusion} contains a concluding discussion.

\section{Overview}\label{sec:summaryintro}

We develop our method using the fermionic kernel $K$; we show in
Appendix \ref{app:bosonic} that in fact this kernel can also be used for
bosonic Green's functions. To simplify the notation, we also
restrict our discussion to scalar-valued Green's functions, as the
extension to the matrix-valued case is straightforward.

We assume the spectral density $\rho$, which may in general be a
distribution, is integrable and supported in
$[-\omax,\omax]$. It is convenient to further nondimensionalize \eqref{eq:lehmann} by performing the change of variables $\tau \gets \tau/\beta$ and $\omega \gets \beta \omega$. 
In these variables, we have $\tau \in [0,1]$, and the support of
$\rho(\omega)$
is contained in $[-\Lambda,\Lambda]$, with $\Lambda = \beta \omax$.
$\Lambda$ is a user-determined parameter. An estimate of $\omax$, and
therefore of $\Lambda$, can often be obtained on physical grounds, but
in general $\Lambda$ is used as an accuracy parameter and is increased
until convergence is reached. Then, assuming $\Lambda$ is taken
sufficiently large, \eqref{eq:lehmann} is equivalent to the
\emph{truncated} Lehmann representation
\begin{equation} \label{eq:tlehmann}
  G(\tau) =  - \int_{-\Lambda}^{\Lambda} K(\tau,\omega)
  \rho(\omega) \, d\omega,
\end{equation}
for $K$ given by (\ref{eq:defK}) with $\beta = 1$.

As for the IR, we exploit the low numerical rank
of an appropriate discretization of $K$ to obtain a compact representation of $G$.
We simply use the ID, rather than the SVD, after discretizing $K$ on
a carefully constructed grid. We will show that 
$G(\tau)$ can be approximated to any fixed accuracy $\epsilon$ by a discrete sum with $r$ terms,
\begin{equation} \label{eq:dlehmann}
  G(\tau) \approx  \gdlr(\tau) \equiv \sum_{k=1}^r  K(\tau,\omega_k)
  \wh{g}_k.
\end{equation}
Here $\{\omega_k(\Lambda, \epsilon)\}_{k=1}^r$ is a collection of selected
frequencies, and
the spectral density $\rho$ has been replaced by a discrete set of
coefficients $\wh{g}_k$. A minus sign has been absorbed into the
coefficients to simplify expressions.
The basis functions of this representation are simply exponentials,
\begin{equation} \label{eq:dlehmann2}
  \gdlr(\tau) = \sum_{k=1}^r \frac{e^{-\omega_k \tau}}{1+e^{-\omega_k}}
  \wh{g}_k = \sum_{k=1}^r  \wt{g}_k e^{-\omega_k \tau},
\end{equation}
a feature which simplifies many calculations. 
We refer to (\ref{eq:dlehmann}, \ref{eq:dlehmann2}) as a discrete
Lehmann representation of $G$. 

We emphasize that given a user-specified error tolerance $\epsilon$ and
a choice of $\Lambda$, the $r$ selected frequencies $\omega_k$ are
universal; that is, independent of $G$. 
Furthermore, $r$, which we refer to as the DLR rank, is close to the
$\epsilon$-rank of $K(\tau,\omega)$, which is the number of IR
basis functions for the same choice of $\Lambda$ and $\epsilon$, so
the DLR also requires at most $\OO{\log (\Lambda)
\log\paren{1/\epsilon}}$ degrees of freedom.
The high energy cutoff $\Lambda$ plays an important role in this
representation, as it controls the regularity of $G(\tau)$,
allowing a representation by a finite combination of exponentials.
We will prove the existence of a representation \eqref{eq:dlehmann} with error tightly
controlled by $\epsilon$, and describe a method to construct such a
representation by interpolation of $G$ at $r$ selected nodes in imaginary time or
Matsubara frequency.

\begin{figure}[t]
  \centering
  \includegraphics[width=0.23\textwidth]{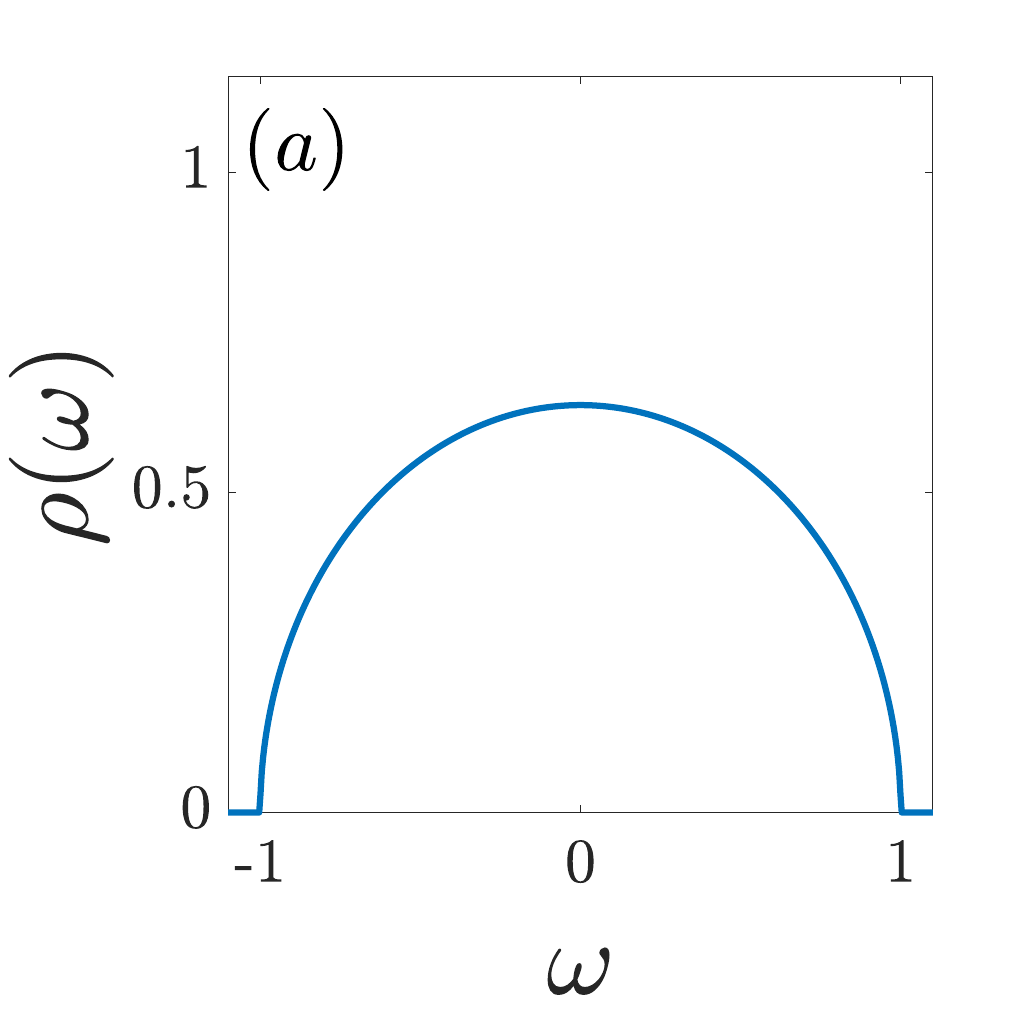}
  \includegraphics[width=0.23\textwidth]{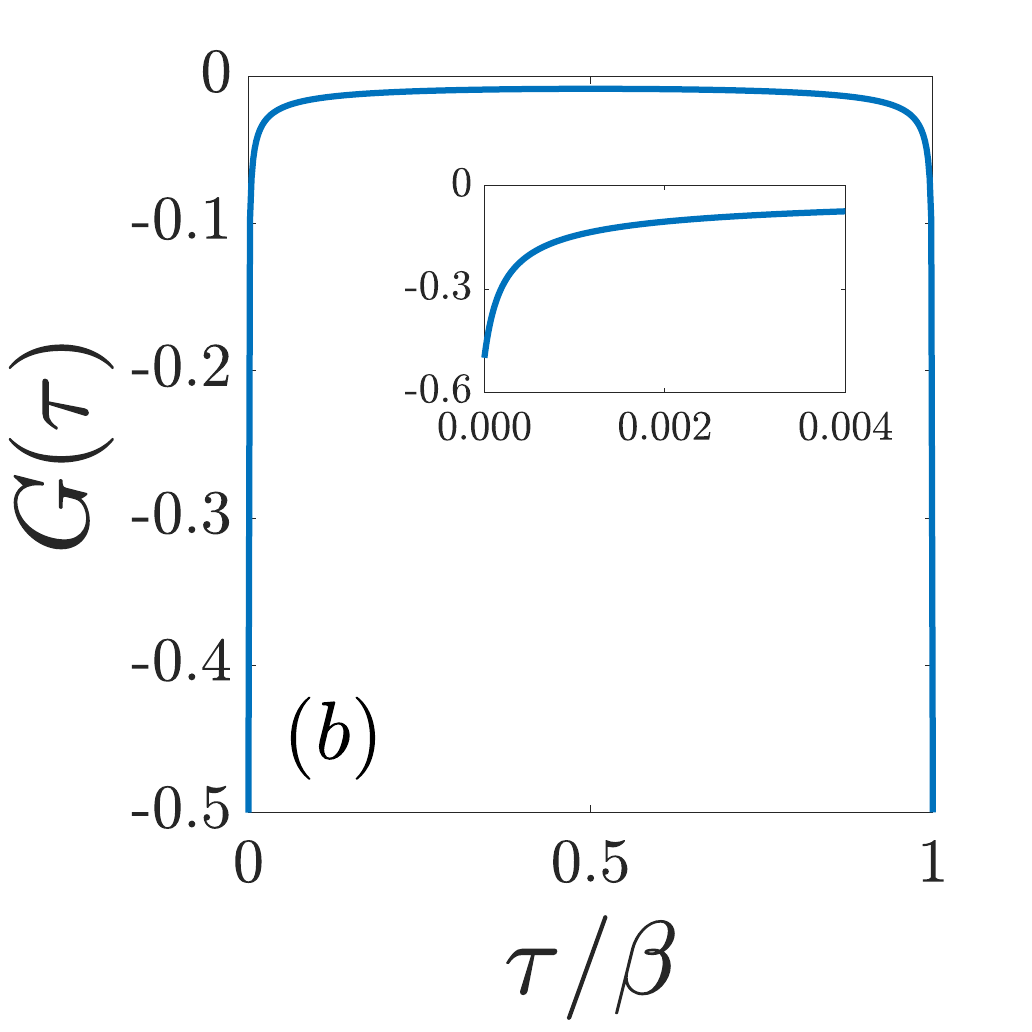}
  
  \includegraphics[width=0.23\textwidth]{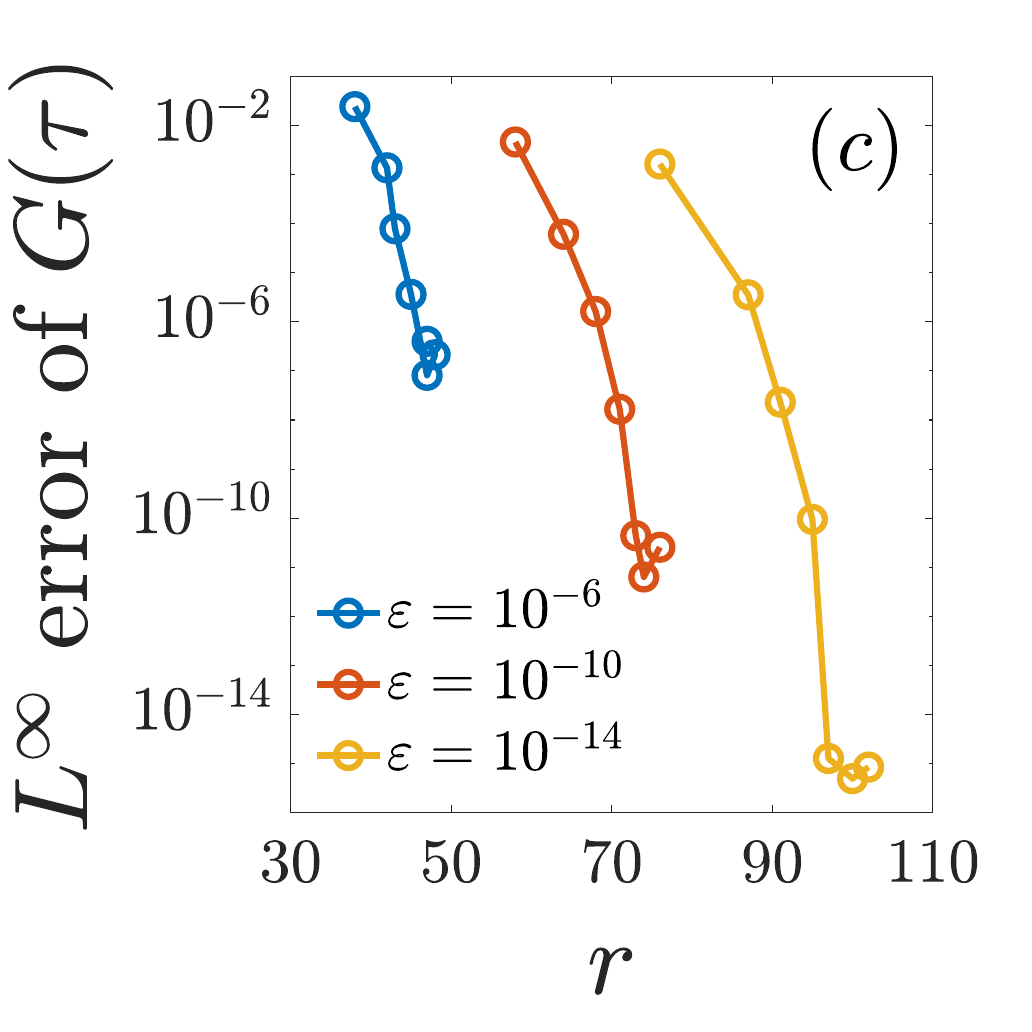}
  \includegraphics[width=0.23\textwidth]{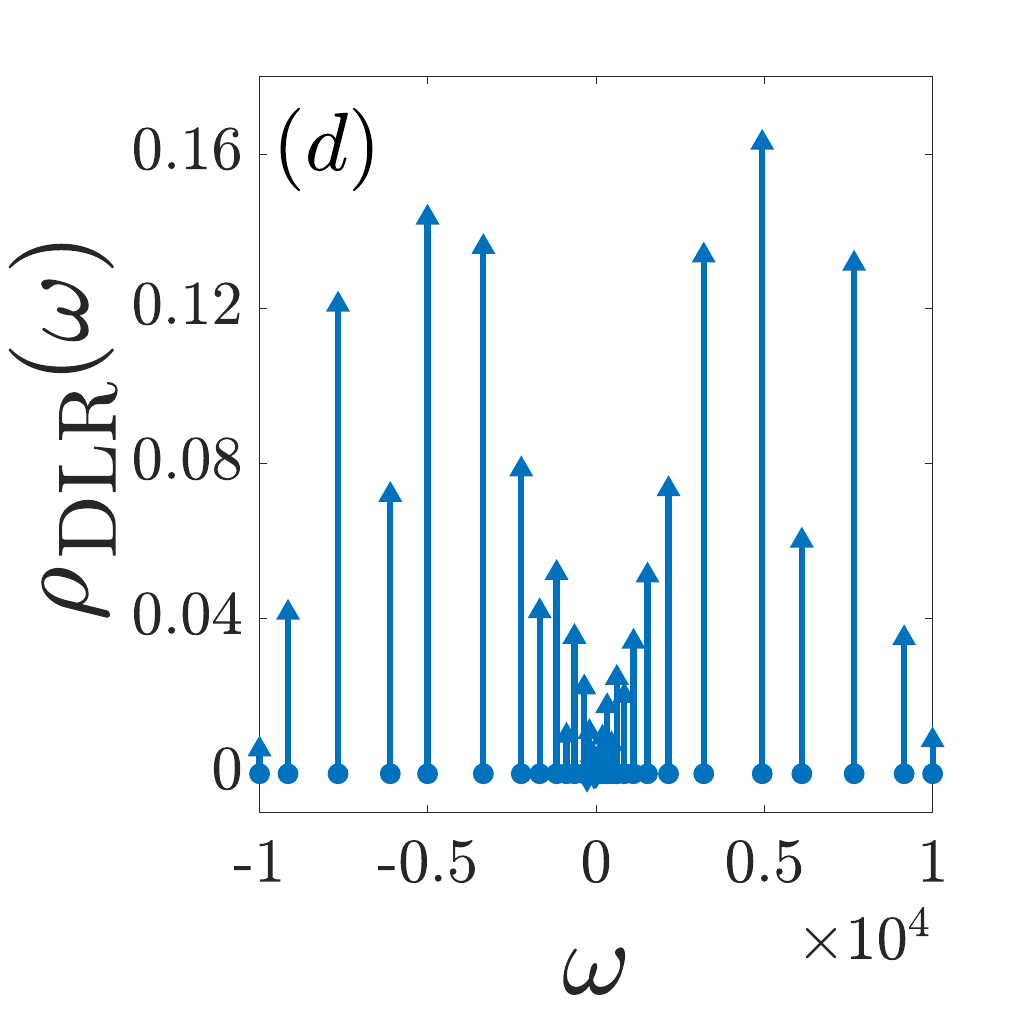}

  \caption{(a) Spectral density $\rho(\omega) =
  \frac{2}{\pi} \sqrt{1-\omega^2} \, \theta(1-\omega^2)$.
  (b) Corresponding imaginary time Green's function $G(\tau)$ with $\beta = 10^4$.
  (c) $\norm{G-\gdlr}_\infty$ as a function of the number of basis
  functions $r$ for $\epsilon= 10^{-6}, 10^{-10}, 10^{-14}$. 
  The values of $r = r(\Lambda, \epsilon)$ correspond to $\Lambda = 0.2 \times 10^4, 0.4 \times
  10^4, \ldots, 1.2 \times 10^4$.
  (d) Representation of the DLR coefficients $\wh{g}_k$ as an effective
  spectral density which is a sum of $\delta$ functions, as in
  \eqref{eq:deltasum}.}
\label{fig:semicirc}
\end{figure}

A first example is presented in Figure \ref{fig:semicirc}.
We take $\beta = 10^4$, and consider a particle-hole symmetric fermionic Green's function $G(\tau)$ defined by 
the spectral density 
$\rho(\omega) = \frac{2}{\pi} \sqrt{1-\omega^2} \, \theta(1-\omega^2)$,
with $\theta$ the Heaviside function, as shown in Figures
\ref{fig:semicirc}a and \ref{fig:semicirc}b.
Figure \ref{fig:semicirc}c shows the error of the DLR (\ref{eq:dlehmann}, \refeq{eq:dlehmann2}) as
a function of $r$ for fixed $\epsilon= 10^{-6}, 10^{-10}, 10^{-14}$.
Here, we vary $\Lambda$ near the known sufficient value of $10^4$ 
($\beta = 10^4$ and $\rho$ is supported in $[-1,1]$)
and plot the error here versus $r(\Lambda)$, instead of $\Lambda$, to emphasize the number of basis functions.
The error decays super-exponentially at first, and reaches $\epsilon$
when $\Lambda \approx 10^4$.
The value of $r$ at which convergence is reached depends on $\epsilon$, so that
in practice, to obtain the smallest possible basis for a given
accuracy, one should first choose $\epsilon$ 
and then increase $\Lambda$ until convergence.

The DLR can be formally interpreted as a spectral representation 
with an effective spectral density $\rhodlr$ which is a sum
of $\delta$ functions:
\begin{equation} \label{eq:deltasum}
  \rho_\text{DLR}(\omega) = -\sum_{k=1}^r \wh{g}_k \delta(\omega-\omega_k).
\end{equation}
Such a representation is made possible by the ill-conditioning of the integral operator
defining the Lehmann representation; up to a fixed precision
$\epsilon$, the spectral density corresponding to a given imaginary
time Green's function is highly non-unique. Thus, we simply pick
one such spectral density with a particularly simple form, rather than
attempting to reconstruct the original spectral density. 
Figure \ref{fig:semicirc}d
shows a graphical
representation of $\rho_\text{DLR}$ and hence of the coefficients
$\wh{g}_k$ and the selected frequencies $\omega_k$.

\section{Mathematical preliminaries} \label{sec:tools}

This section will review our two main numerical tools: composite
Chebyshev interpolation, which will be used to obtain an accurate
initial discretization of the kernel $K(\tau,\omega)$, and the
interpolative decomposition, which will be used for low rank
compression. 

\subsection{Composite Chebyshev interpolation} \label{sec:chebyshev}

Polynomial interpolation at Chebyshev nodes is a well-conditioned method
for the approximation of a smooth function $f$ on an interval.
\cite{trefethen19} If $f$ can be analytically continued to a
neighborhood of $[a,b]$, the error of the interpolant in the supremum
norm decreases geometrically with its degree, and if $f$ can be
analytically continued to the entire complex plane, the convergence is
super-geometric; see Ref. \onlinecite{trefethen19} (Thm. 8.2). There are fast and stable algorithms
to evaluate Chebyshev interpolants, such as the method of
barycentric Lagrange interpolation. \cite{berrut04,higham04}

For functions with sharp features or variation at multiple length
scales, using a single polynomial interpolant on $[a,b]$ is inefficient.
A better alternative is to construct a piecewise polynomial interpolant
by the method of \emph{composite Chebyshev interpolation at fixed order}. To be precise,
let $[a_1,b_1], [a_2,b_2], \ldots,
[a_n,b_n]$ with $a = a_1 < b_1 = a_2 < b_2 = \cdots = a_n < b_n = b$ be a
collection of subintervals partitioning $[a,b]$. Let $\{x_{ij}\}_{i=1}^p$ be the
$p$ Chebyshev nodes on $[a_j,b_j]$. Then
$\{x_{ij}\}_{{i,j=1}}^{p,n}$ is called a \emph{composite
Chebyshev grid}. Let $\ell_{ij}(x)$ be the \emph{Lagrange
polynomial} corresponding to the $i$th grid point on the $j$th panel;
this is the polynomial of degree $p-1$ which satisfies
\[\ell_{ij}(x) = 
\begin{cases}
  1 &\text{if } x = x_{ij} \\
  0 &\text{if } x = x_{kj}, k \neq i.
\end{cases}
\]
Let $\chi_j(x)$ be the
characteristic function on the interval $[a_j,b_j]$. Then the degree
$p-1$ composite Chebyshev interpolant of a function $f$ on $[a,b]$
corresponding to the above partition is given by
\begin{equation} \label{eq:interp}
  \wh{f}(x) = \sum_{j=1}^n \chi_j(x) \sum_{i=1}^p \ell_{ij}(x) f(x_{ij}).
\end{equation}
Evidently, we have $f(x_{ij}) = \wh{f}(x_{ij})$ for each $i=1,\ldots,p$
and $j=1,\ldots,n$. The partition of $[a,b]$ should be chosen to resolve
local features of $f$, and the degree $p$ should be chosen sufficiently
large so that the rapidly converging Chebyshev interpolants of $f$ on
each subinterval $[a_j,b_j]$ are accurate.

To simplify expressions, we define the truncated Lagrange
polynomial on the interval $[a_j,b_j]$ by $\wb{\ell}_{ij} \equiv \ell_{ij}
\chi_j$.  It will also sometimes be convenient to cast the double index
$i=1,\ldots,p$, $j=1,\ldots,n$ for the composite grid points to a single
index $i = 1,\ldots,p \times n$, with $x_i \gets x_{ij}$, $\ell_i \gets
\ell_{ij}$, and $\wb{\ell}_i \gets \wb{\ell}_{ij}$. In this notation,
\eqref{eq:interp} becomes
\begin{equation} \label{eq:interp2}
  \wh{f}(x) = \sum_{i=1}^{p \times n} \wb{\ell}_{i}(x) f(x_i).
\end{equation}

\subsection{Interpolative decomposition}

We say an $m \times n$ matrix $A$ is numerically low rank,
or more specifically, has low $\epsilon$-rank, if $A$ has only $r \ll \min(m,n)$ singular
values larger than $\epsilon$. The best rank $r$ approximation of $A$
in the spectral norm is given by its SVD
truncated to the first $r$ singular values, and its error in that norm
is the next singular value $\sigma_{r+1}$; see Ref.
\onlinecite{ballani16} (Sec. 2, Thm. 2). Thus, the truncated SVD (TSVD)
yields an approximation with error $\epsilon$ in the spectral norm
for a matrix with $\epsilon$-rank $r$.

The interpolative decomposition is an alternative to the
TSVD for compressing numerically low rank matrices. It has the
advantage that the column space is represented by selected columns of
$A$, rather than an orthogonalization of the columns of $A$, as in the
TSVD. The price is a mild and controlled loss of optimality compared
with the TSVD. The ID and related algorithms are described in Refs.
\onlinecite{cheng05,liberty07,gu96}; in particular, we make use of the form of
the ID and the theoretical results summarized in Ref. \onlinecite{liberty07}.

Given $A \in \CC^{m \times n}$, the rank $r$ ID is given by
\[A \approx BP\]
with $B \in \CC^{m \times r}$ a matrix containing $r$ selected columns
of $A$, and $P \in \CC^{r \times n}$, the so-called projection matrix,
containing the coefficients required to approximately recover all of the
columns of $A$ from the $r$ selected columns. The error of the
decomposition is given by
\begin{equation} \label{eq:idest}
  \norm{A-BP}_2 \leq \sqrt{r(n-r)+1} \, \sigma_{r+1},
\end{equation}
so the ID gives a rank $r$ approximation of $A$ which is at most a factor of
$\sqrt{r(n-r)+1}$ less accurate than 
the TSVD. The numerical stability
of the ID as a representation of $A$ can also be guaranteed; in
particular, we have
\begin{equation} \label{eq:idpest}
  \norm{P}_2 \leq \sqrt{r (n-r) + 1}.
\end{equation}
The references given above contain detailed statements of the relevant
results which we have quoted here, along with the accompanying analysis.

Numerical algorithms are available which construct such a decomposition
with bounds typically within a small factor of those stated above. The standard
algorithm, described in Ref. \onlinecite{cheng05}, proceeds in two
steps. First, the pivoted QR process is applied to $A$, yielding a
collection of $r$ columns of $A$ -- corresponding to the pivot indices --
which are, in a certain sense, as close as possible to being mutually
orthogonal. These $r$ columns comprise the matrix $B$ in the ID.
Next, a linear system is solved to determine the coefficients of the
remaining columns of $A$ in the basis determined by $B$. These
coefficients are stored in the matrix $P$. 
The cost of this algorithm is $\OO{r m n}$. If the rank $r$ is not known
a priori, it is straightforward to apply this algorithm in a
rank-revealing manner, 
so that given an input $\epsilon$ it
yields an estimated $\epsilon$-rank $r$ and a rank $r$ ID with
$\norm{A - BP}_2 \leq \epsilon$. Of course, the returned
$\epsilon$-rank may be larger than the true $\epsilon$-rank,
consistent with the suboptimality of the estimate \eqref{eq:idest} and
the behavior of the singular values of $A$.

We remark that for several of the algorithms presented in this article
-- in particular for all the algorithms involving the DLR -- we only
ever need to perform the pivoted QR step of the ID to identify $k$
selected columns of a matrix, and in particular do not
need to construct the full ID. Nevertheless, the
presentation in terms of the ID is both conceptually and theoretically
useful, and helps to unify our discussions of the DLR and the IR, so we
adopt that language throughout. Our descriptions of algorithms in the
text will make this point clear.

The Fortran library \texttt{ID} provides an implementation of the ID
algorithm. \cite{idlib,iddoc} A Python interface is available in SciPy.
\cite{idlibscipy} For our numerical experiments, we use the implementation
of the rank-revealing pivoted QR algorithm contained in the Fortran
version of the library.

\section{Discrete Lehmann representation} \label{sec:dlr}

The DLR basis functions are built by a two-step procedure. First, we
discretize $K(\tau,\omega)$ on a composite Chebyshev fine grid
$\{(\tau_i^f,\omega_j^f)\}_{i=1,j=1}^{M,N}$, obtaining a
matrix with entries $K(\tau_i^f,\omega_j^f)$. Then, we obtain a small subset $\{\omega_{l}\}_{l=1}^r$ of the
fine grid points in $\omega$ from the ID of this matrix, such that
\begin{equation} \label{eq:Kinterp}
  K(\tau,\omega) \approx \sum_{l=1}^r K(\tau,\omega_l) \pi_l(\omega)
\end{equation}
holds to high accuracy uniformly in $\tau$, for some coefficients
$\pi_l(\omega)$. The functions
$\{K(\tau,\omega_l)\}_{l=1}^r$ are referred to as the \emph{DLR basis
functions}. Inserting \eqref{eq:Kinterp} into the Lehmann
representation \eqref{eq:lehmann} will establish the existence of the
DLR. The discretization of $K$ will be discussed in Section
\ref{sec:kdisc}, and the construction of the DLR basis in Section
\ref{sec:basis}.
In Sections \ref{sec:dlrpts} and \ref{sec:matpts}, we will describe a stable method of
constructing the DLR of a Green's function $G$ from samples of $G$ at
only $r$ selected imaginary time and Matsubara frequency nodes,
respectively. In Section \ref{sec:algsummary} we will give a practical
summary of the various procedures, and we will demonstrate the DLR with
a few simple examples in Section \ref{sec:numex}. Throughout the
discussion, except when describing specific physical examples, we will
work in the nondimensionalized variables described at the beginning of
Section \ref{sec:summaryintro}, with $\tau \in [0,1]$, $\omega \in
[-\Lambda,\Lambda]$, and $K(\tau,\omega) = e^{-\tau
\omega}/(1+e^{-\omega})$.

\subsection{Discretization of $K(\tau,\omega)$} \label{sec:kdisc}

We discretize $K(\tau,\omega)$ by finding grids
sufficient to resolve $K(\tau,\omega_0)$ on $\tau \in [0,1]$ for all
fixed $\omega_0 \in [-\Lambda,\Lambda]$, and $K(\tau_0,\omega)$ on
$\omega \in [-\Lambda,\Lambda]$ for all fixed $\tau_0 \in [0,1]$. A
closely related problem was considered in Ref. \onlinecite{gimbutas20}, in which
it is shown (Lemma 4.4) that all exponentials in the family $\{e^{-\omega \tau}\}_{\omega \in
[1,\Lambda]}$ can be represented to error uniformly less than $\epsilon$
on $\tau \geq 0$ in a basis of $\OO{\log(\Lambda)
\log\paren{1/\epsilon}}$ exponentials chosen from the family. As in their
proof, we will make use of dyadically refined composite Chebyshev grids.
A minor modification of their proof is sufficient to give a rigorous
justification of our method, though we do not discuss the details
here. 

We begin with the first case, for $\omega_0 \in [0,\Lambda]$, which gives 
$K(\tau,\omega_0) = c e^{-\omega_0 \tau}$
for a constant $c$; a family of decaying exponentials.
Consider the composite Chebyshev grid on $\tau \in [0,1]$ \emph{dyadically
refined} towards the origin; that is, with intervals given by $a_1 = 0$, $a_i = b_{i-1} = 2^{-(m-i+1)}$ for
$i=2,\ldots,m$, and $b_m = 1$. We take $m \sim \log_2 \Lambda$ to
resolve the smallest length scale in the family of exponentials, which
appears for $\omega_0 = \Lambda$.
With this choice, the degree parameter $p$ can be chosen sufficiently large so that the resulting
composite Chebyshev interpolant is uniformly accurate for any $\Lambda$.
Double precision machine accuracy $\epsm$ can be achieved with a moderate
choice of $p$, since the Chebyshev interpolants of the exponentials converge
rapidly with $p$. The accuracy of the interpolants can be checked
directly, and $p$ refined to convergence.

For $\omega_0 \in [-\Lambda,0]$, we observe that $K(\tau,\omega_0) = c
e^{\omega_0 (1-\tau)}$, revealing a symmetry in $K$ about $\tau = 1/2$. We therefore split the last interval
$[1/2,1]$ in our partition into a set of subintervals dyadically
refined towards $\tau = 1$, in the same manner as above. The resulting
composite Chebyshev grid is sufficient to resolve $K(\tau,\omega_0)$ for
all $\omega_0 \in [-\Lambda,\Lambda]$, and contains $\OO{\log \Lambda}$
points. An example of such a grid is shown in Figure \ref{fig:grids}a.

\begin{figure}[t]
  \centering
  \includegraphics[width=.46\textwidth]{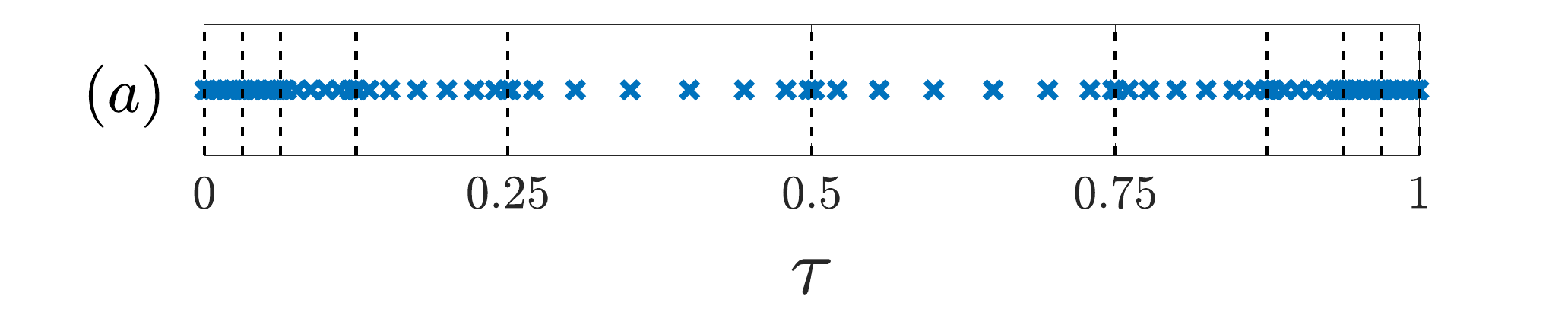}
  \includegraphics[width=.46\textwidth]{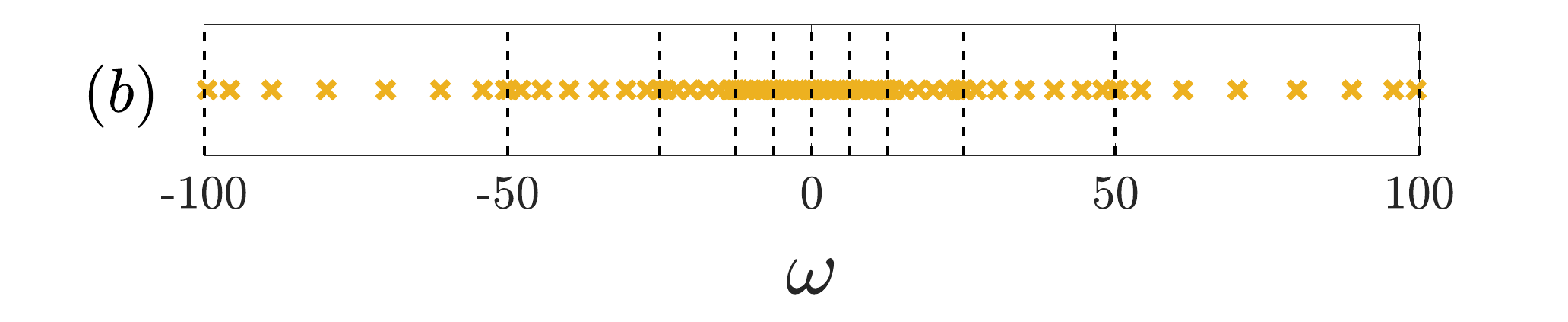}
  \includegraphics[width=.46\textwidth]{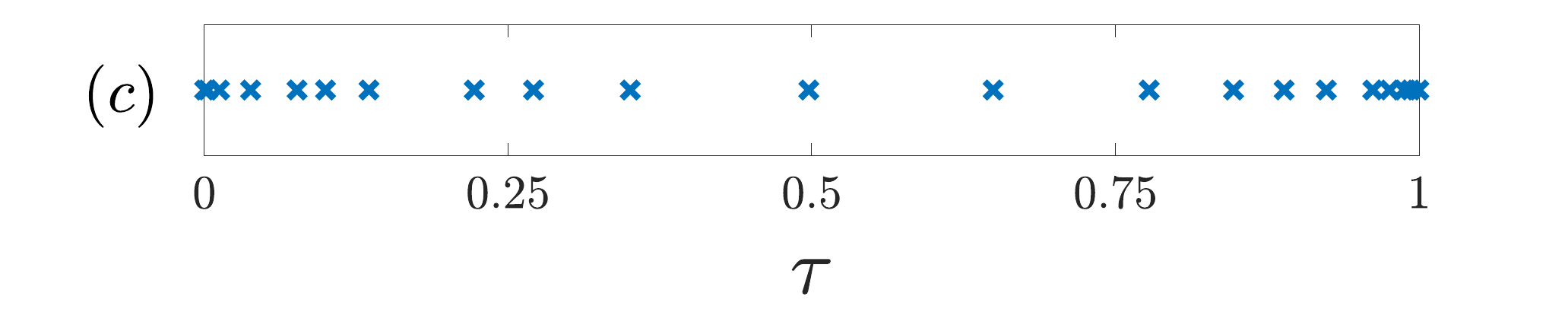}
  \includegraphics[width=.46\textwidth]{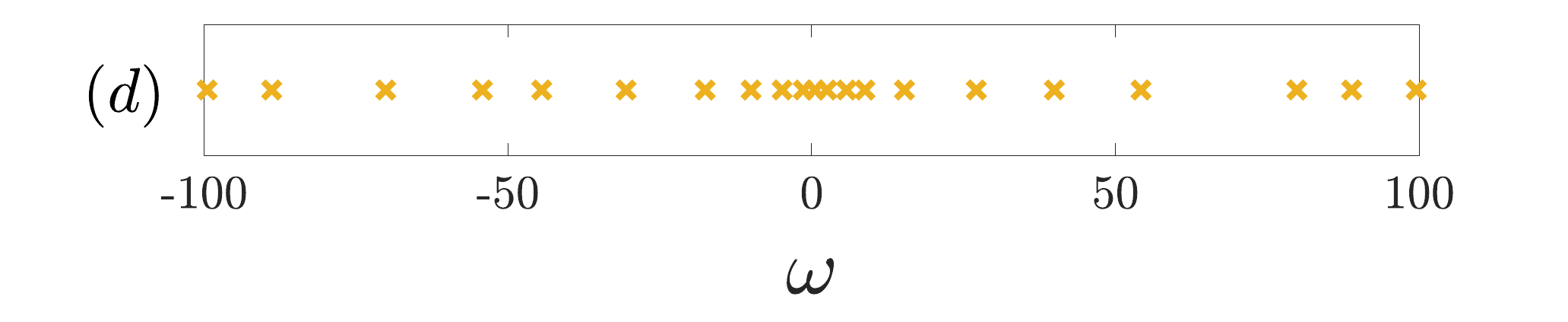}
  \caption{(a) Fine grid points $\tau_i^f$ for $p = 8$ and $n = 5$.
  Subinterval endpoints are indicated by the dashed lines. (b) Fine grid points $\omega_k^f$ for $p = 8$ and $n = 5$.
  (c) The $21$
  imaginary time DLR nodes selected from the fine grid in (a) for $\Lambda = 100$
  and $\epsilon = 10^{-6}$. For readability, we have used a smaller value of $p$ here
  than we do in practice. (d) The $21$ DLR frequencies selected from the fine grid in (b). }
\label{fig:grids}
\end{figure}

We refer to the nodes of this composite Chebyshev grid as the set of
\emph{fine grid points in $\tau$}, and denote them by
$\{\tau_j^f\}_{j=1}^M$, using the single-index notation for a composite
Chebyshev grid. Here, $M = p \times m$, where $m$ is the total number of
subintervals in the partition of $[0,1]$. Thus we can ensure that for each fixed $\omega
\in [-\Lambda,\Lambda]$, the composite Chebyshev interpolant on the fine
grid in $\tau$ is uniformly accurate to $\epsm$;
using the notation
defined in Section \ref{sec:chebyshev}, we have
\begin{equation} \label{eq:kinterptau}
  \norm{K(\tau,\omega) - \sum_{i=1}^M \wb{\ell}_i(\tau)
K(\tau_i^f,\omega)}_\infty < \epsm
\end{equation}
with $M = \OO{\log \Lambda}$.

We next consider fixed $\tau_0 \in [0,1]$, for which we have $K(\tau_0,\omega) =
e^{-\tau_0 \omega}/\paren{1+e^{-\omega}}$. This is again a family of
functions which are sharply peaked near the origin, and we discretize
$[-\Lambda,\Lambda]$ by a composite Chebyshev grid with intervals
dyadically refined towards the origin from the positive and negative
direction until the smallest panels are of unit size, which again
requires $n \sim \log_2 \Lambda$. A similar choice of $p$ is again sufficient
to obtain accuracy $\epsm$ for any $\Lambda$. An example of this grid is
shown in Figure \ref{fig:grids}b. 

The resulting \emph{fine grid points in $\omega$} are denoted by 
$\{\omega_j^f\}_{k=1}^N$, and give a composite Chebyshev interpolant of
$K(\tau,\omega)$ on $\omega \in [-\Lambda,\Lambda]$ for each $\tau$
which is uniformly accurate to $\epsm$; that is
\begin{equation} \label{eq:kinterpomega}
  \norm{K(\tau,\omega) - \sum_{j=1}^N 
  K(\tau,\omega_j^f) \wb{\ell}_j(\omega)}_\infty < \epsm
\end{equation}
with $N = \OO{\log \Lambda}$.
We note an abuse of notation: $\wb{\ell}_i(\tau)$
refers to the truncated Lagrange polynomials for the fine grid
in $\tau$, whereas $\wb{\ell}_j(\omega)$ refers to those for the fine
grid in $\omega$. Combining \eqref{eq:kinterptau} and
\eqref{eq:kinterpomega}, and possibly increasing $p$, we obtain
\begin{equation} \label{eq:kinterptauomega}
  \norm{K(\tau,\omega) - \sum_{i=1}^M \sum_{j=1}^N \wb{\ell}_i(\tau) 
  K(\tau_i^f,\omega_j^f) \wb{\ell}_j(\omega)}_\infty < \epsm.
\end{equation}

We summarize as follows. The kernel $K(\tau,\omega)$ may be
represented by composite Chebyshev interpolants of $M$ and $N$ terms in
$\tau$ and $\omega$, respectively, with subintervals chosen by
dyadically subdivision. These
representations can be constructed at a negligible cost, and directly
checked for accuracy. We have $M = p m$ and $N = p n$; in
practice, we find $m = n = \max\paren{\log_2 \Lambda,1}$ and $p = 24$ to be
sufficient to ensure double precision machine accuracy.

For simplicity of exposition, we will assume in the remainder of the article that the
interpolation errors in \eqref{eq:kinterptau}, \eqref{eq:kinterpomega},
and \eqref{eq:kinterptauomega}
are identically zero. Indeed, given these estimates, $K(\tau,\omega)$ is
indistinguishable from its interpolants to the machine precision, and we
can just as well take the interpolants as our definition of $K$.

\subsection{The DLR basis} \label{sec:basis}

Define $A \in \RR^{M \times N}$ with entries given by $A_{ij} =
K(\tau_i^f,\omega_j^f)$. Figure \ref{fig:ranks}a shows the singular values of
$A$ for a few choices of $\Lambda$. Evidently, the singular
values decay at least exponentially, so that for each fixed
$\Lambda$, the $\epsilon$-rank of $A$ is $\OO{\log(1/\epsilon)}$.
Figure \ref{fig:ranks}b shows that the rate of exponential decay
is proportional to $\log(\Lambda)$. It
follows that the $\epsilon$-rank is $\OO{\log(\Lambda)
\log(1/\epsilon)}$. A derivation and analysis of this bound will be given
in a forthcoming publication. \cite{chen_inprep}

Since the column space of $A$ characterizes the
subspace of imaginary time Green's functions defined by
\eqref{eq:tlehmann}, the low numerical rank of $A$ shows that this subspace is
finite-dimensional to a good approximation.
An equivalent observation is made
in Ref. \onlinecite{shinaoka17}, where it justifies
using the left singular vectors of a discretization of $K(\tau,\omega)$ 
as a compressed representation of
imaginary time Green's functions. This is the IR basis, which we discuss
in detail in Section \ref{sec:ir}.

\begin{figure}[t]
  \centering
  \includegraphics[width=.23\textwidth]{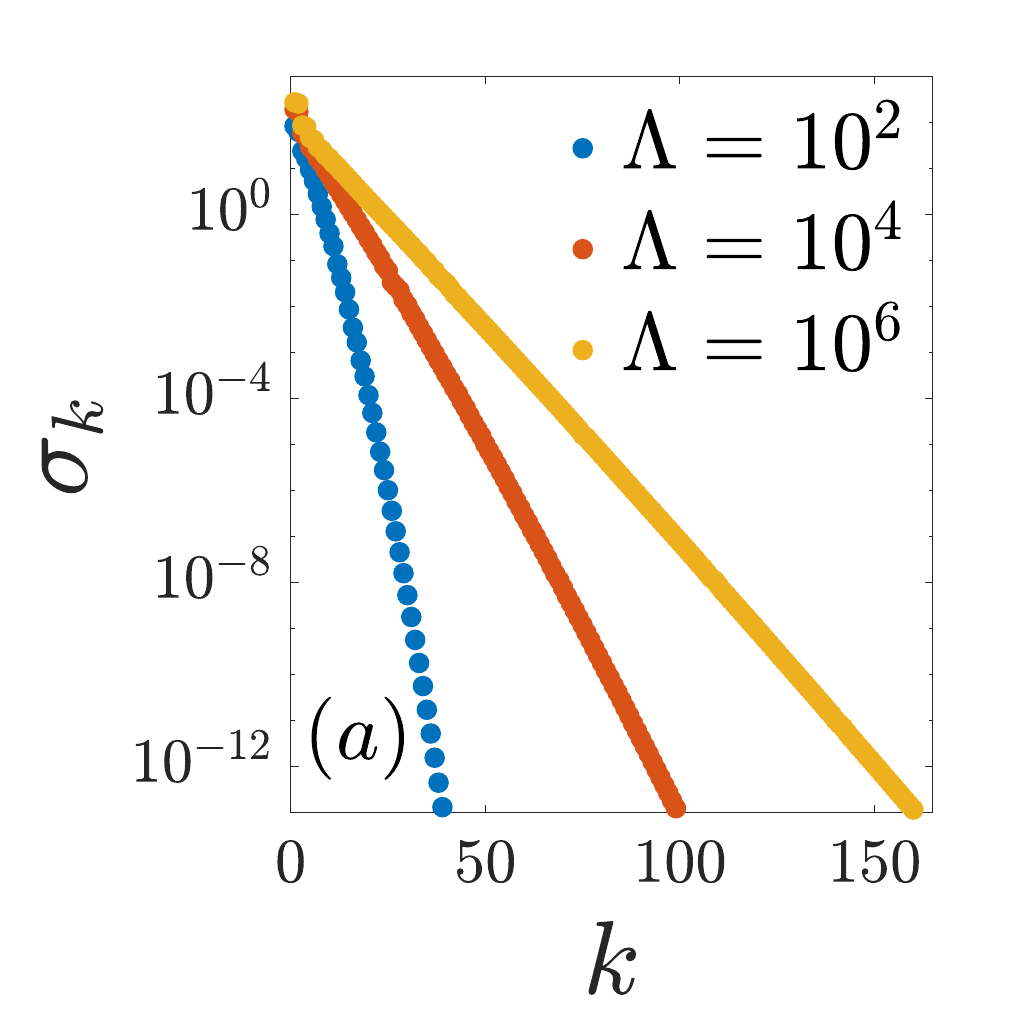}
  \includegraphics[width=.23\textwidth]{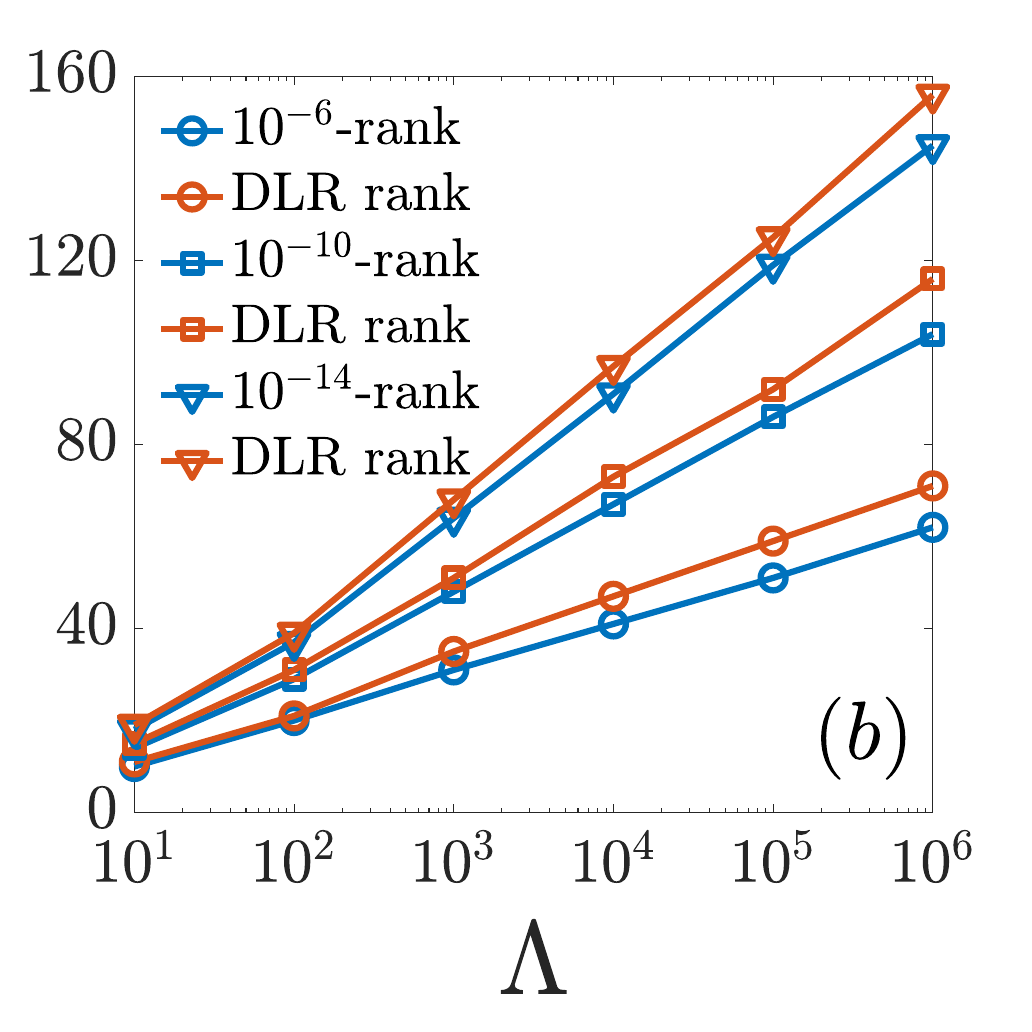}
  \caption{(a) Singular values $\sigma_k$ of the matrix $A_{ij} =
  K(\tau_i^f,\omega_j^f)$, for various $\Lambda$. (b) $\epsilon$-rank
  of $A$ against $\Lambda$ for various $\epsilon$ (blue), and the DLR rank
  (number of DLR basis functions) for the same choice of $\epsilon$
  (orange).}
\label{fig:ranks}
\end{figure}

Here, we use the ID to build a basis for
the column space of $A$. Let $\epsilon$ be a user-provided error tolerance.
We can construct a rank $r$ ID of $A$, 
\begin{equation} \label{eq:Aid}
  A = BP + \errmat,
\end{equation}
for $B \in \RR^{M \times r}$, $P \in \RR^{r
\times N}$, and $\errmat \in \RR^{M \times N}$ an error matrix with 
\[\norm{\errmat}_2 \leq \epsilon.\]
It follows from
\eqref{eq:idest} and the rapid decay of the singular values of $A$ that
$r$ will be at worst only slightly larger than the true
$\epsilon$-rank of $A$. The discrepancy is shown in Figure
\ref{fig:ranks}b,
with the blue points showing the true $\epsilon$-rank $r$ against $\Lambda$
for several $\epsilon$, and the orange points showing $r$ as obtained by
the ID with the same choices of $\epsilon$, which we refer to as the DLR
rank. This is a useful figure to refer to, as it shows the number of DLR
basis functions required to represent any imaginary time Green's
function obeying a high energy cutoff $\Lambda$ to a given $\epsilon$
accuracy.

Writing \eqref{eq:Aid} entrywise gives
\[K(\tau_i^f,\omega_j^f) = \sum_{l=1}^r K(\tau_i^f,\omega_l) P_{lj} +
\errmat_{ij}\]
for a subset $\{\omega_l\}_{l=1}^r$ of $\{\omega_j^f\}_{j=1}^N$. This
subset corresponds to the selected columns in the ID, and we refer to it
as the collection of \emph{DLR frequencies}. 
Summing both sides against $\wb{\ell}_i(\tau)$ and $\wb{\ell}_j(\omega)$
gives
\[K(\tau,\omega) = \sum_{l=1}^r K(\tau,\omega_l) \pi_l(\omega) +
\errmat(\tau,\omega)\]
with $\pi_l(\omega) = \sum_{j=1}^n  P_{lj} \wb{\ell}_j(\omega)$ and
$\errmat(\tau,\omega) = \sum_{i=1}^M \sum_{j=1}^N \wb{\ell}_i(\tau) \errmat_{ij}
\wb{\ell}_j(\omega)$. Inserting this into \eqref{eq:tlehmann}, we obtain
\begin{equation} \label{eq:dlrvalid}
  \begin{multlined}
  G(\tau) = -\sum_{l=1}^r K(\tau,\omega_l) \int_{-\Lambda}^\Lambda
\pi_l(\omega) \rho(\omega) \, d\omega \\ - \int_{-\Lambda}^\Lambda \errmat(\tau,\omega) \rho(\omega) \, d\omega.
  \end{multlined}
\end{equation}
Letting $\wh{g}_l = -\int_{-\Lambda}^\Lambda \pi_l(\omega) \rho(\omega) \,
d\omega$ gives our first main result. The bound on the error
term is proven in Appendix \ref{app:thm1}.
\begin{theorem} \label{thm:dlr}
  Suppose $G$ is given by its truncated Lehmann representation \eqref{eq:tlehmann}.
  Then there exist coefficients $\{\wh{g}_l\}_{l=1}^r$ such that
  \begin{equation} \label{eq:dlrwerr}
    G(\tau) = \sum_{l=1}^r K(\tau,\omega_l) \wh{g}_l + \errfun(\tau)
  \end{equation}
  with $\{\omega_l\}_{l=1}^r$ chosen corresponding to the selected
  columns of the ID \eqref{eq:Aid}. The error term $\errfun(\tau)$ satisfies
  \[\norm{\errfun}_\infty \leq c \epsilon \norm{\rho}_1\]
  for a constant $c$ which depends only on $p$, the Chebyshev degree
  parameter defined above.
\end{theorem}
The constant $c$ is mild and computable; for $p = 24$, it is less than
$10$. 
The $r$ functions $K(\tau,\omega_l)$ are referred to as the DLR basis
functions, and are characterized solely by the DLR frequencies $\omega_l$ selected
in the ID. An example of a set of DLR frequencies, selected from the fine grid
shown in Figure \ref{fig:grids}b with $\Lambda = 100$ and $\epsilon =
10^{-6}$, is shown in Figure \ref{fig:grids}d.

We note that in practice, it is not necessary to form the full ID in
order to obtain the DLR basis, since
we do not use the projection matrix $P$. Rather, we only need to
identify the DLR frequencies $\{\omega_l\}_{l=1}^r$. The selection of
the DLR frequencies takes place in the pivoted QR step of the ID
algorithm. Thus to construct the DLR basis, we simply apply the
rank-revealing pivoted QR algorithm to the columns of $A$ with a tolerance
$\epsilon$.

\subsection{The imaginary time DLR grid} \label{sec:dlrpts}

In general, the spectral density $\rho$ is not known a priori, so we cannot find the coefficients
$\wh{g}_l$ in \eqref{eq:dlrwerr} using the construction above. Rather, we
will identify a set of $r$
\emph{imaginary time interpolation nodes} $\tau_k$ so that expansion
coefficients can be recovered from the values $g_k = G(\tau_k)$ by solving an interpolation
problem using the basis functions $K(\tau,\omega_l)$.

Consider the matrix $B \in \RR^{M \times r}$ introduced above, with entries $B_{il} =
K(\tau_i^f,\omega_l)$. Forming the ID of $B^T$ gives
\begin{equation} \label{eq:Bid}
  B = R \kmat,
\end{equation}
with $\kmat \in \RR^{r \times r}$ consisting of selected rows of $B$, and $R
\in \RR^{M \times r}$ the associated projection matrix. The $r$ selected
rows of $B$ correspond to a subset $\{\tau_k\}_{k=1}^r$ of the fine grid
points $\{\tau_i^f\}_{i=1}^M$ in imaginary time, which we refer to as
the imaginary time DLR grid. We have
\begin{equation} \label{eq:kmatdef}
  \kmat_{kl} = K(\tau_k,\omega_l).
\end{equation}
Writing \eqref{eq:Bid} entrywise and summing over the truncated Lagrange
polynomials in $\tau$, we obtain
\begin{equation} \label{eq:krecover}
  K(\tau,\omega_l) = \sum_{k=1}^r \gamma_k(\tau) K(\tau_k,\omega_l)
  \equiv \sum_{k=1}^r \gamma_k(\tau) \kmat_{kl}
\end{equation}
with $\gamma_k(\tau) = \sum_{i=1}^M \wb{\ell}_i(\tau) R_{ik}$.
Equation \eqref{eq:krecover} tells us that the DLR basis functions can be
recovered from their values at the imaginary time DLR grid points. It will follow that
a Green's function can similarly be recovered
from its values on this grid. An example of an
imaginary time DLR grid, selected from the
fine grid shown in Figure \ref{fig:grids}a with $\Lambda = 100$ and
$\epsilon = 10^{-6}$, is
shown in Figure \ref{fig:grids}c.

The recovery may be carried out in practice by computing the values $g_k
= G(\tau_k)$ for $k=1,\ldots,r$, solving the interpolation problem
\begin{equation} \label{eq:ginterp}
  g = \kmat \wh{g}
\end{equation}
for {\it DLR coefficients} $\wh{g}_k$, and using
\begin{equation} \label{eq:gapprox}
  \gdlr(\tau) = \sum_{l=1}^r K(\tau,\omega_l) \, \wh{g}_k
\end{equation}
as an approximation of $G$. Here, $g, \wh{g} \in \RR^r$.
Although it is tempting to compare \eqref{eq:gapprox} with
\eqref{eq:dlrwerr} and assume $\gdlr \approx G$ holds to high accuracy,
this is not guaranteed a priori. Indeed, if the interpolation nodes $\tau_k$
were not selected carefully, this would not be the case. However, the
following stability result, proven in Appendix \ref{app:thm2}, leads to an accuracy guarantee.
\begin{lemma} \label{lem:dlrstability}
  Suppose $\gdlr$ and $\hdlr$ are given by
  \[\gdlr(\tau) = \sum_{l=1}^r K(\tau,\omega_l) \wh{g}_l\]
  and 
  \[\hdlr(\tau) = \sum_{l=1}^r K(\tau,\omega_l) \wh{h}_l,\]
  respectively, with $\{\omega_l\}_{l=1}^r$ chosen as above. Let $g, h
  \in \RR^r$ be given by $g_k = \gdlr(\tau_k)$,
  $h_k = \hdlr(\tau_k)$, with $\{\tau_k\}_{k=1}^r$ the imaginary time DLR grid
  determined by the selected rows of the ID \eqref{eq:Bid}. Then
  \[\norm{\gdlr-\hdlr}_\infty \leq \sqrt{2} \norm{R}_2 \norm{g-h}_2. \]
\end{lemma}
The ID guarantees that $\norm{R}_2$ is controlled; in particular, we
have the estimate \eqref{eq:idpest}. Since $M = \OO{\log \Lambda}$
and $r$ is small, this factor in the estimate is small in
practice.
With Lemma \ref{lem:dlrstability} in hand, we consider the following
practical question: if a Green's function $G$ is sampled at the DLR grid
points with some error, how accurate is the approximation $\gdlr$ given
by \eqref{eq:gapprox}, with the coefficients $\rho_l$ obtained by
solving the interpolation problem \eqref{eq:ginterp}?

\begin{theorem} \label{thm:accuracy}
Let $G$ be a Green's function given by a truncated Lehmann representation
  \eqref{eq:tlehmann}.  Let $g \in \RR^r$ be a vector of samples of $G$ at
  the imaginary time DLR grid points $\tau_k$, up to an error $\eta \in \RR^r$: $g_k =
  G(\tau_k) + \eta_k$. Suppose
  $\wh{g} \in \RR^r$ solves the corresponding interpolation problem \eqref{eq:ginterp}
  up to a residual error $\alpha$: $\kmat \wh{g} - g = \alpha$, with
  $\alpha \in \RR^r$. Let $\gdlr$ be given by
  \[\gdlr(\tau) = \sum_{l=1}^r K(\tau,\omega_l) \wh{g}_l.\]
  Then
  \begin{multline*}
  \norm{G-\gdlr}_\infty \leq c \paren{1+\sqrt{2r} \norm{R}_2}
  \norm{\rho}_1 \epsilon \\ + \sqrt{2} \norm{R}_2 \paren{\norm{\eta}_2
  + \norm{\alpha}_2}
  \end{multline*}
  with $c$ the constant from Theorem \ref{thm:dlr}.
\end{theorem}
\begin{proof}
  Theorem \ref{thm:dlr} guarantees that
  \[G(\tau)= \hdlr(\tau) + \errfun(\tau)\]
  for $\hdlr$ a DLR expansion and $\errfun$ a
  controlled error. We also have that
  \[\gdlr(\tau_k) = g_k + \alpha_k = G(\tau_k) + \eta_k + \alpha_k.\]
  These expressions, and Lemma \ref{lem:dlrstability}, give
  \begin{align*}
    \norm{G-\gdlr}_\infty &= \norm{\errfun+\hdlr-\gdlr}_\infty \\
    &\leq \norm{\errfun}_\infty + \sqrt{2} \norm{R}_2
  \norm{\{\errfun(\tau_k)\}_{k=1}^r +
  \eta + \alpha}_2 \\
    &\begin{multlined} \leq  \paren{1+\sqrt{2r} \norm{R}_2}
      \norm{\errfun}_\infty \\+ \sqrt{2}
    \norm{R}_2 \paren{\norm{\eta}_2 + \norm{\alpha}_2}.\end{multlined}
  \end{align*}
The result follows from the bound on $\norm{\errfun}_\infty$ given in
  Theorem \ref{thm:dlr}.
\end{proof}
It is expected, and our numerical experiments
confirm, that typically $\norm{\alpha}_2 \approx \norm{\eta}_2$.
Thus the accuracy of the approximation \eqref{eq:gapprox} is indeed
determined by the user-input error tolerance $\epsilon$, and is
limited only by the accuracy to which $G$ can be evaluated. 
We remark that this holds true despite the fact that the matrix
$\kmat$ is
ill-conditioned, and therefore that the computed DLR coefficients
$\wh{g}_l$ are not expected to be close to those appearing in
\eqref{eq:dlrwerr}. Indeed, this ill-conditioning reflects a fundamental
non-uniqueness in $\wh{g}_l$. However, it will not prevent a standard linear
solver from identifying a solution with small residual, and therefore
does not imply any difficulty in accurately representing $G$.

\subsection{DLR in the Matsubara frequency domain} \label{sec:matpts}

A DLR can be transformed to the Matsubara frequency domain analytically.
Indeed, we have
\begin{equation} \label{eq:kn}
  K(i\nu_n,\omega) = \int_0^1 K(\tau,\omega) e^{-i \nu_n \tau} d \tau
= \paren{\omega + i \nu_n}^{-1},
\end{equation}
with Matsubara frequency points
\[i \nu_n = 
\begin{cases}
  i (2n+1) \pi & \text{for fermionic Green's functions} \\
  i 2n \pi & \text{for bosonic Green's functions.}
\end{cases}
\]
A DLR expansion $G(\tau) = \sum_{l=1}^r K(\tau,\omega_l) \wh{g}_l$ therefore transforms to the Matsubara frequency domain as 
\[G(i \nu_n) = \sum_{l=1}^r K(i \nu_n,\omega_l) \wh{g}_l.\]

We can construct a set of \emph{Matsubara frequency interpolation nodes} using the ID.
As in the previous section, we simply apply the ID to the rows of the
matrix with entries $K(i\nu_n,\omega_l)$, for $n =
-\nmax,\ldots,\nmax$, and $l = 1,\ldots,r$. Here $\nmax$ is a chosen
Matsubara frequency cutoff. This process returns $r$ selected Matsubara frequency
interpolation nodes $i\nu_{n_k}$. As before, it is not necessary to form the full ID, but only
to use the pivoted QR algorithm to identify the selected nodes. The DLR coefficients can be recovered by solving
the interpolation problem
\begin{equation} \label{eq:interpmatsu}
  G(i\nu_{n_k}) = \sum_{l=1}^r K(i\nu_{n_k},\omega_l) \wh{g}_l,
\end{equation}
for $k=1,\ldots,r$, which is analogous to \eqref{eq:ginterp}.
One must ensure that the
Matsubara frequency nodes have been converged with respect to $\nmax$, and
in practice we find $\nmax \sim \Lambda$ is usually a sufficient
cutoff.

This procedure requires carrying out the pivoted QR algorithm on the
rows of a
$2\nmax+1 \times r$ matrix, and typically $\nmax = \OO{\Lambda}$.
It is 
more expensive than the procedure to select the imaginary time DLR grid
points, which uses the pivoted QR algorithm on an $M \times r$ matrix, with $M =
\OO{\log \Lambda}$. However, it is still quite fast in
practice for
moderate values of $\Lambda$. If it were to become a bottleneck, one
could design a more efficient scheme to select the Matsubara frequency
interpolation nodes from a smaller subset of the full Matsubara frequency
grid $-\nmax \leq n \leq \nmax$.

\subsection{Summary of DLR algorithms} \label{sec:algsummary}

We pause to summarize the practical procedures we have described to
build and work with the DLR.

\paragraph{Construction of the DLR basis}

To construct the DLR basis for a given choice of $\Lambda$ and
$\epsilon$, we first discretize the kernel
$K(\tau,\omega)$ on a composite Chebyshev grid to obtain the matrix with entries $A_{ij} =
K(\tau_i^f,\omega_j^f)$. We then apply the pivoted QR algorithm, with an
error tolerance $\epsilon$, to the
columns of $A$. The pivots correspond to a set of $r$ DLR frequencies $\omega_l$,
where $r$, the so-called DLR rank, is the number of basis functions
required to represent the full subspace characterized by the truncated
Lehmann integral operator to an accuracy
approximately $\epsilon$. The DLR
basis functions are simply given by $\{K(\tau,\omega_l)\}_{l=1}^r$.

\paragraph{DLR from imaginary time values}

To obtain the $r$ imaginary time interpolation nodes $\tau_k$, we simply
apply the pivoted QR algorithm to the rows of the matrix with entries
$B_{il} = K(\tau_i^f,\omega_l)$. The pivots correspond to the
interpolation nodes. To obtain the DLR coefficients
$\wh{g}_l$ of a Green's function $G(\tau)$, we compute the $r$ values
$g_k = G(\tau_k)$ and solve the $r
\times r$ interpolation problem \eqref{eq:ginterp}.

\paragraph{DLR from Matsubara frequency values}

To obtain the $r$ Matsubara frequency interpolation nodes $i\nu_{n_k}$, we apply
the pivoted QR algorithm in the same manner to the rows of the matrix with entries
$K(i\nu_{n},\omega_l)$, where $-\nmax \leq n \leq \nmax$ for some choice
of $\nmax$. In practice, we find $\nmax = \Lambda$ to be
sufficient in most cases, but $\nmax$ can be increased until the selected
Matsubara frequency nodes no longer change. To obtain the DLR expansion
coefficients $\wh{g}_l$ of a Green's function $G(i\nu_n)$ in the Matsubara
frequency domain, we solve the interpolation problem
\eqref{eq:interpmatsu}.

\paragraph{Transforming between imaginary time and Matsubara frequency
domains}

The DLR coefficients for the representation of a given Green's
function in the imaginary time and Matsubara frequency domains are the
same; one simply takes the Fourier transform of the DLR in imaginary
time explicitly using \eqref{eq:kn} to obtain the DLR in Matsubara
frequency, and inverts the transform explicitly to go in the opposite
direction. Thus, having obtained DLR coefficients for a
Green's function, the representation can be evaluated in either domain.

\paragraph{A remark on the selection of $\Lambda$ and $\epsilon$}
\label{par:epsandlambda}

In our framework, both $\Lambda$ and $\epsilon$ are user-determined parameters
which control the accuracy of a given representation, and each choice of
$\Lambda$ and $\epsilon$ yields some basis of $r$ functions
which should then all be used. This is different from many typical
methods, like orthogonal polynomial approximation, in which one simply
converges a given calculation with respect to the number $m$ of basis
functions directly.
The inclusion of such a user-determined accuracy parameter $\epsilon$ is a
desirable feature of many modern algorithms used in scientific
computing, which enables automatic data compression with an accuracy guarantee.

In practice, to obtain a desired accuracy with the smallest possible number of basis functions, one should choose $\epsilon$ according to that desired accuracy, and not
smaller. One should then converge with respect to $\Lambda$, which
describes the frequency content of the problem, and is therefore
more analogous to the parameter $m$ in the Legendre polynomial method.
This process is illustrated, for example, by Figure \ref{fig:err_sc_1e2},
which is discussed in the next subsection.

\subsection{Numerical examples} \label{sec:numex}

We can test the algorithms described above by evaluating a known Green's
function on the imaginary time or Matsubara frequency DLR grids,
recovering the corresponding DLR coefficients, and measuring the accuracy of the
resulting DLR expansion by computing its error against
$G(\tau)$. We use fermionic Green's functions for all examples.

We first test the imaginary time sampling approach using the Green's
function corresponding to the spectral
density $\rho(\omega) = \frac{2}{\pi} \sqrt{1-\omega^2} \theta\paren{1-\omega^2}$. We
fix $\epsilon$, and measure the $L^\infty$ error of the computed DLR for several
choices of $\Lambda$.
Results for $\beta = 10^4$ were already presented in Figure
\ref{fig:semicirc}c, in which we plot error against the number $r$ of basis functions
obtained using $\Lambda = 0.2 \times 10^{4}, 0.4 \times 10^{4}, \ldots, 1.2
\times 10^{4}$, for $\epsilon = 10^{-6}$, $10^{-10}$, and $10^{-14}$.
We observe rapid convergence with $r$ to error $\epsilon$ in each
case.

In Figures \ref{fig:err_sc_1e2} and \ref{fig:err_sc_1e6}, respectively, we
present similar plots for $\beta = 10^2$ and $\beta = 10^6$. In Figures
\ref{fig:err_sc_1e2}c and \ref{fig:err_sc_1e6}c, we plot the error against $\Lambda$
directly. These plots demonstrate the method as it is used in practice;
$\epsilon$ and $\Lambda$, not $r$, are chosen directly by the user in our
framework. It can be seen from Figures \ref{fig:err_sc_1e2}b and
\ref{fig:err_sc_1e6}b that
choosing $\epsilon$ to be smaller than the actual desired accuracy
simply yields a larger basis than is needed, as was discussed in
Section \ref{sec:algsummary}\,e.

\begin{figure*}[t]
  \centering
  \includegraphics[width=.23\textwidth]{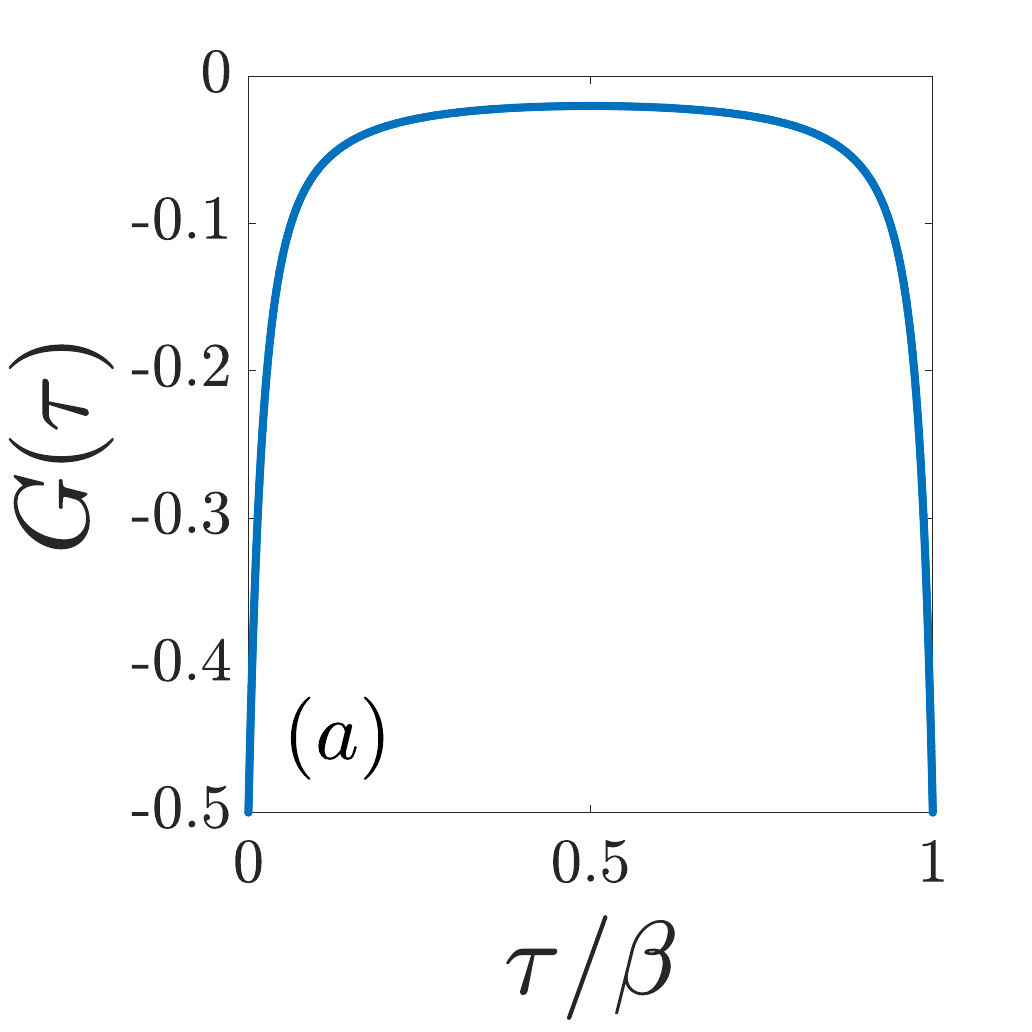}
  \includegraphics[width=.23\textwidth]{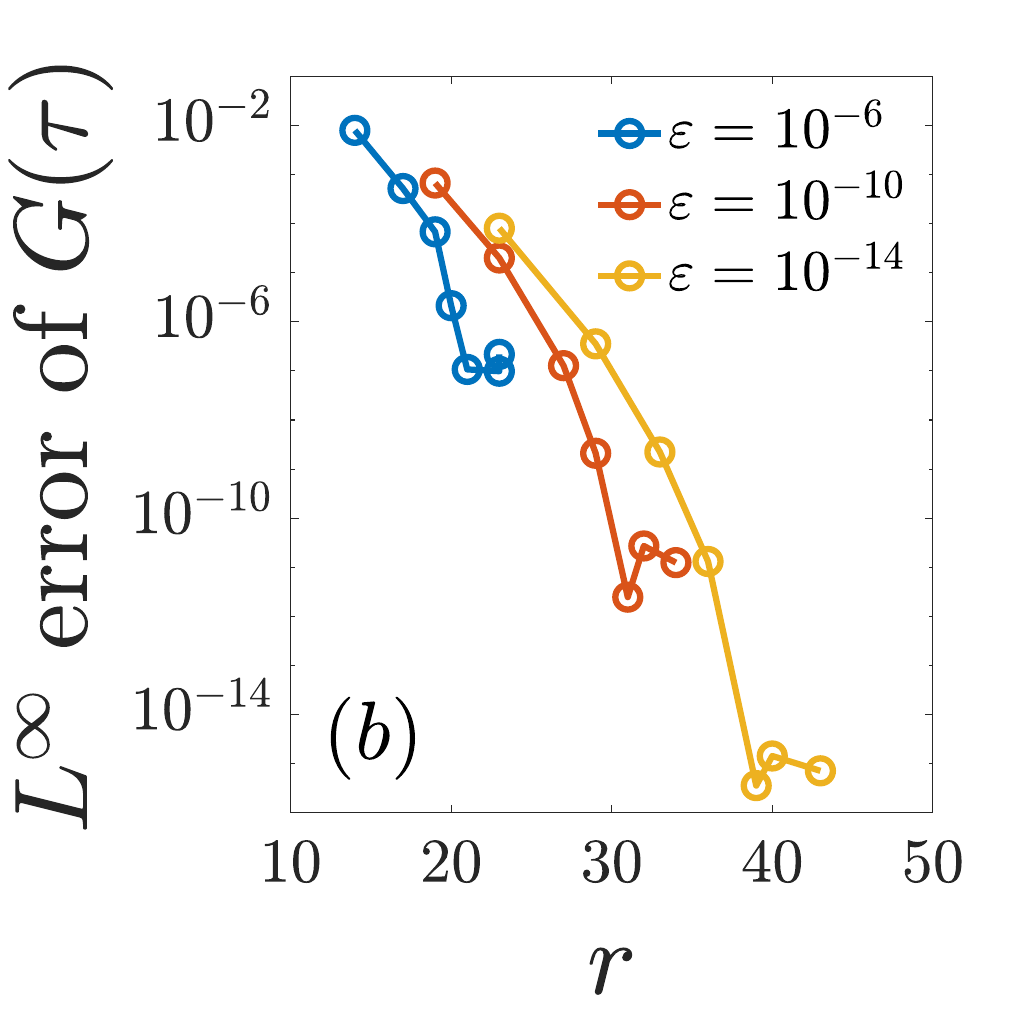}
  \includegraphics[width=.23\textwidth]{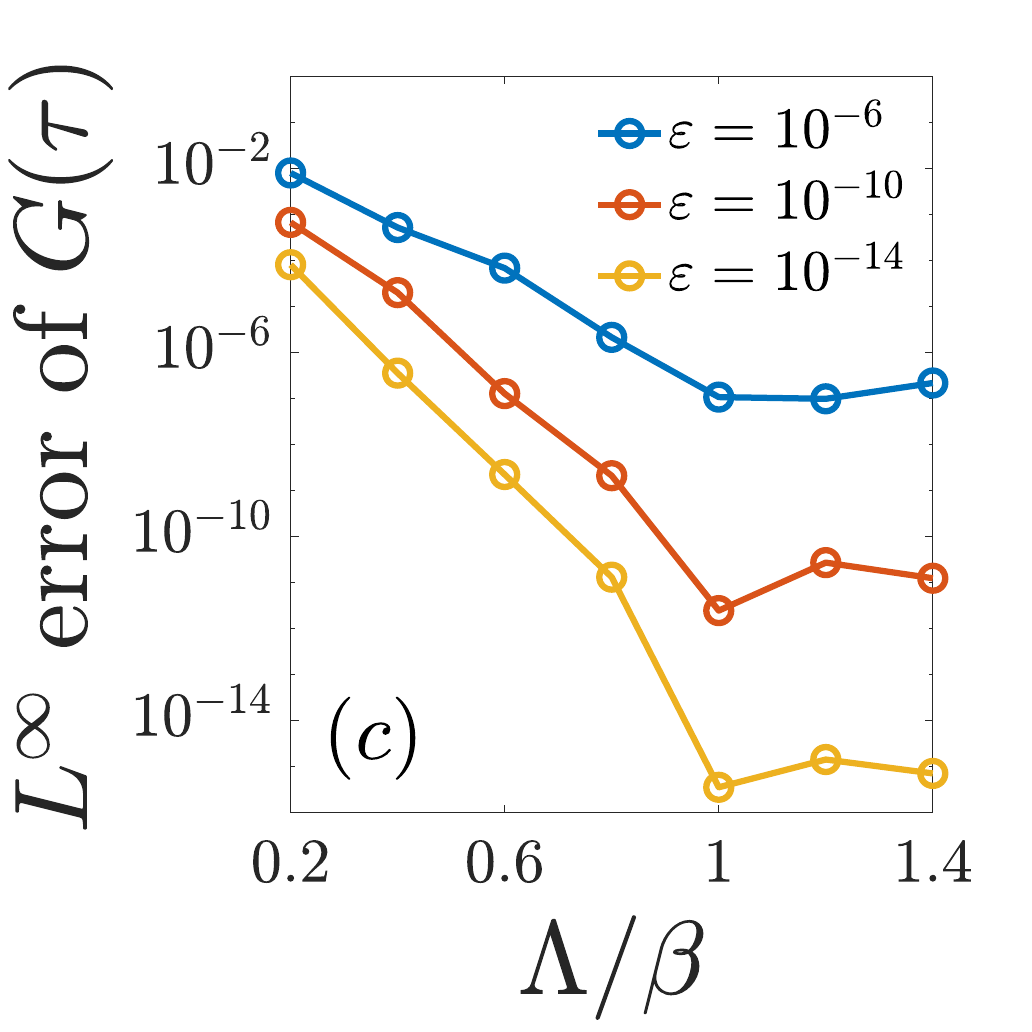}

  \caption{$L^\infty$ error of the DLR approximation of $G(\tau) =
  -\frac{2}{\pi} \int_{-1}^1 K(\tau,\omega) \sqrt{1-\omega^2} \, d\omega$ obtained
  using imaginary time sampling for $\beta = 10^2$ and several choices
  of $\epsilon$. (a) $G(\tau)$. (b) Error versus $r$, the number of basis
  functions. (c) Error versus $\Lambda$.}
  \label{fig:err_sc_1e2}
\end{figure*}

\begin{figure*}[t]
  \centering
  \includegraphics[width=.23\textwidth]{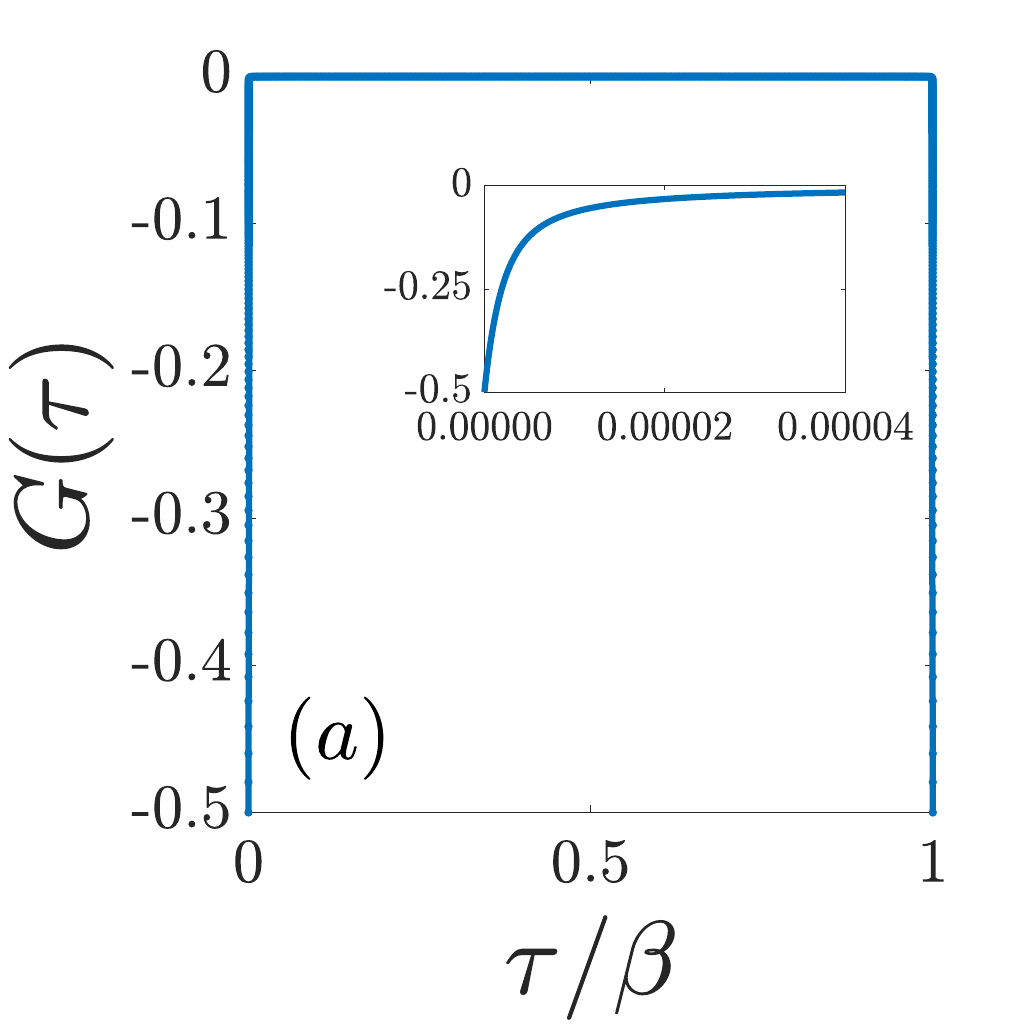}
  \includegraphics[width=.23\textwidth]{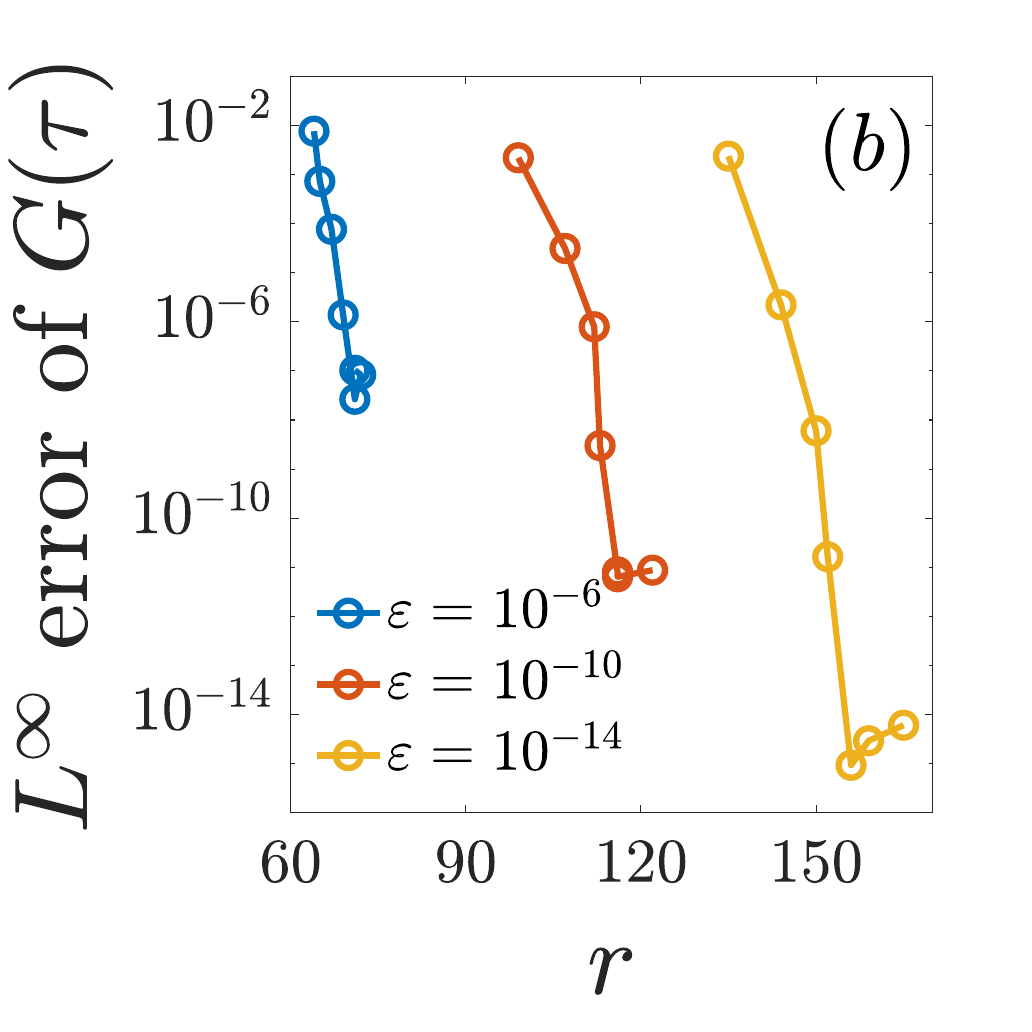}
  \includegraphics[width=.23\textwidth]{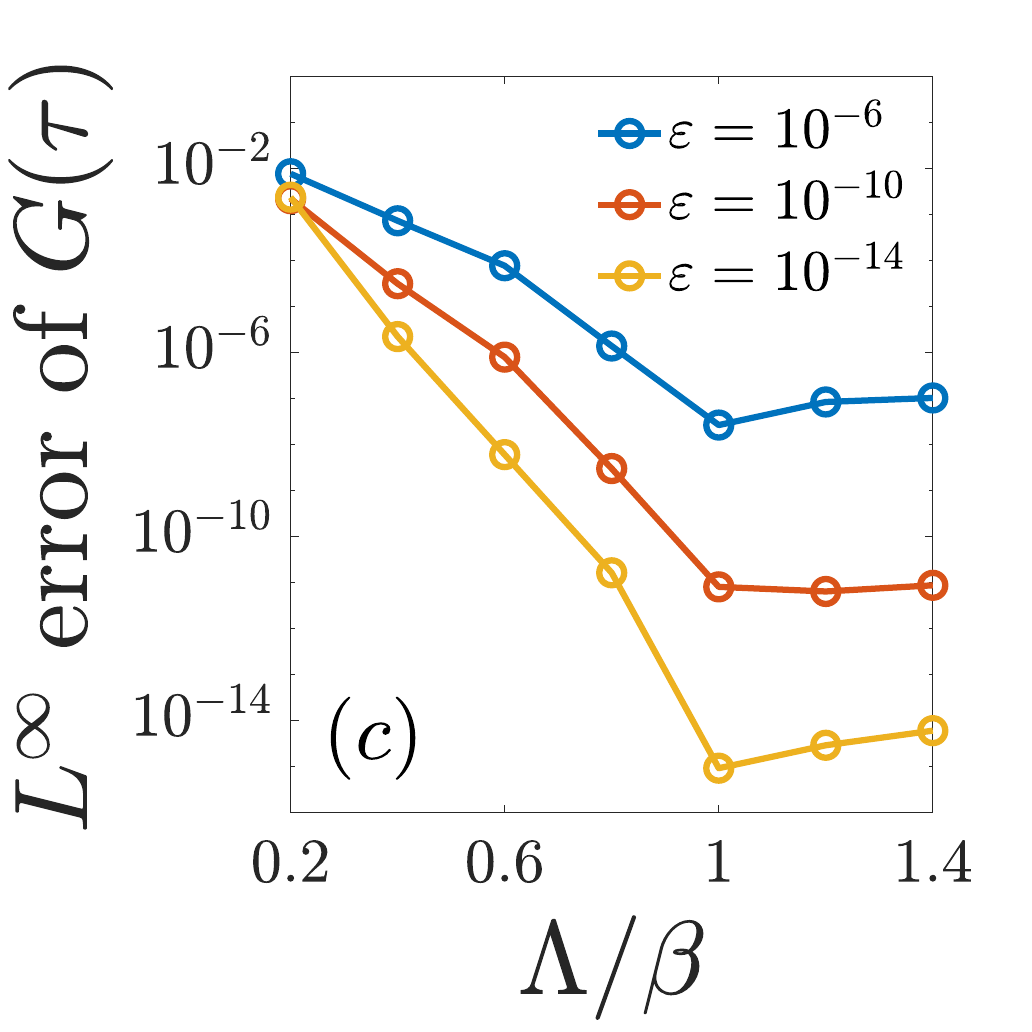}

  \caption{The same as in Figure \ref{fig:err_sc_1e2}, with $\beta
  = 10^6$.}
  \label{fig:err_sc_1e6}

\end{figure*}

We next repeat the experiment using $\rho(\omega) = \paren{\delta(-1/3) +
\delta(1)}/2$ for $\beta = 100$. The Green's function is shown in Figure
\ref{fig:err_ha_1e2}a, and the error versus $r$
in Figure \ref{fig:err_ha_1e2}b. The results are similar to those for
the previous example. We note that the same experiments with $\beta = 10^4$ and
$\beta = 10^6$, and $\Lambda$ adjusted accordingly, give the expected results.

\begin{figure*}[t]
  \centering
  \includegraphics[width=.23\textwidth]{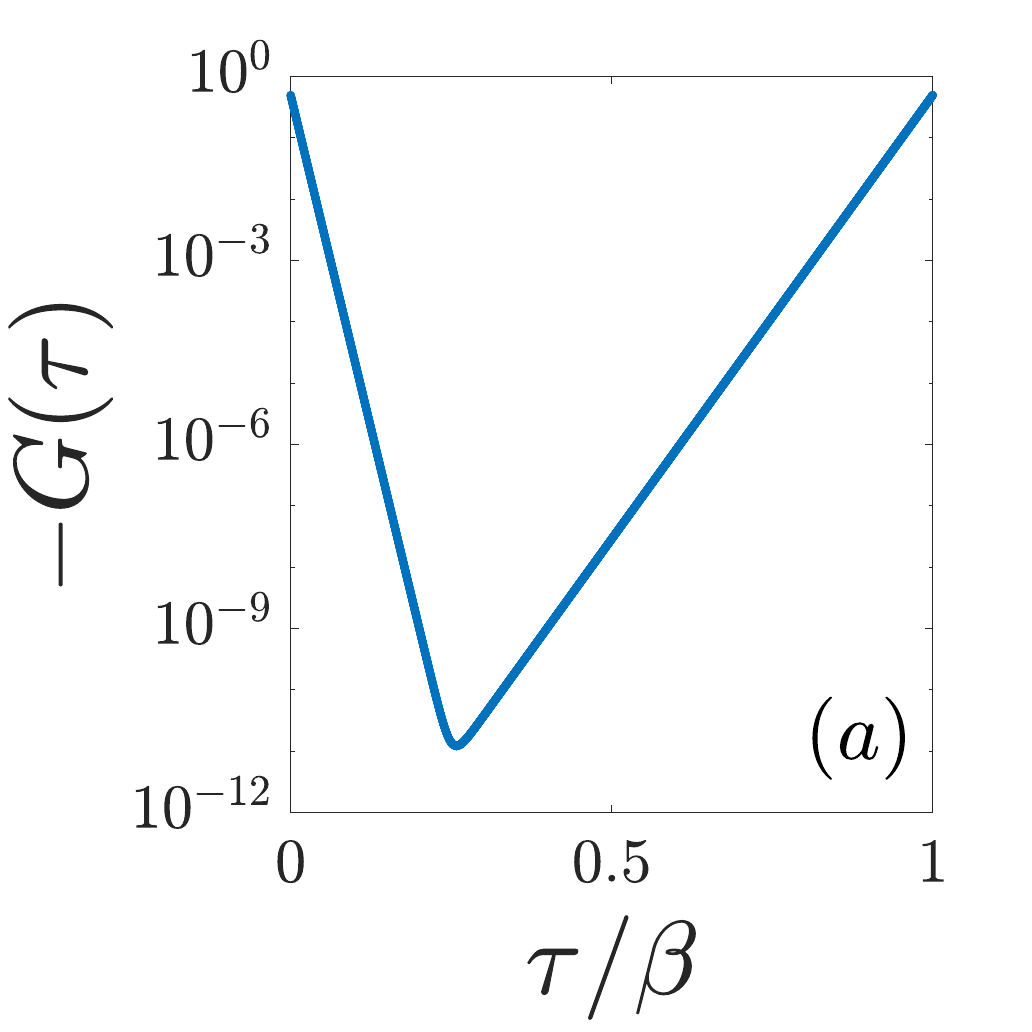}
  \includegraphics[width=.23\textwidth]{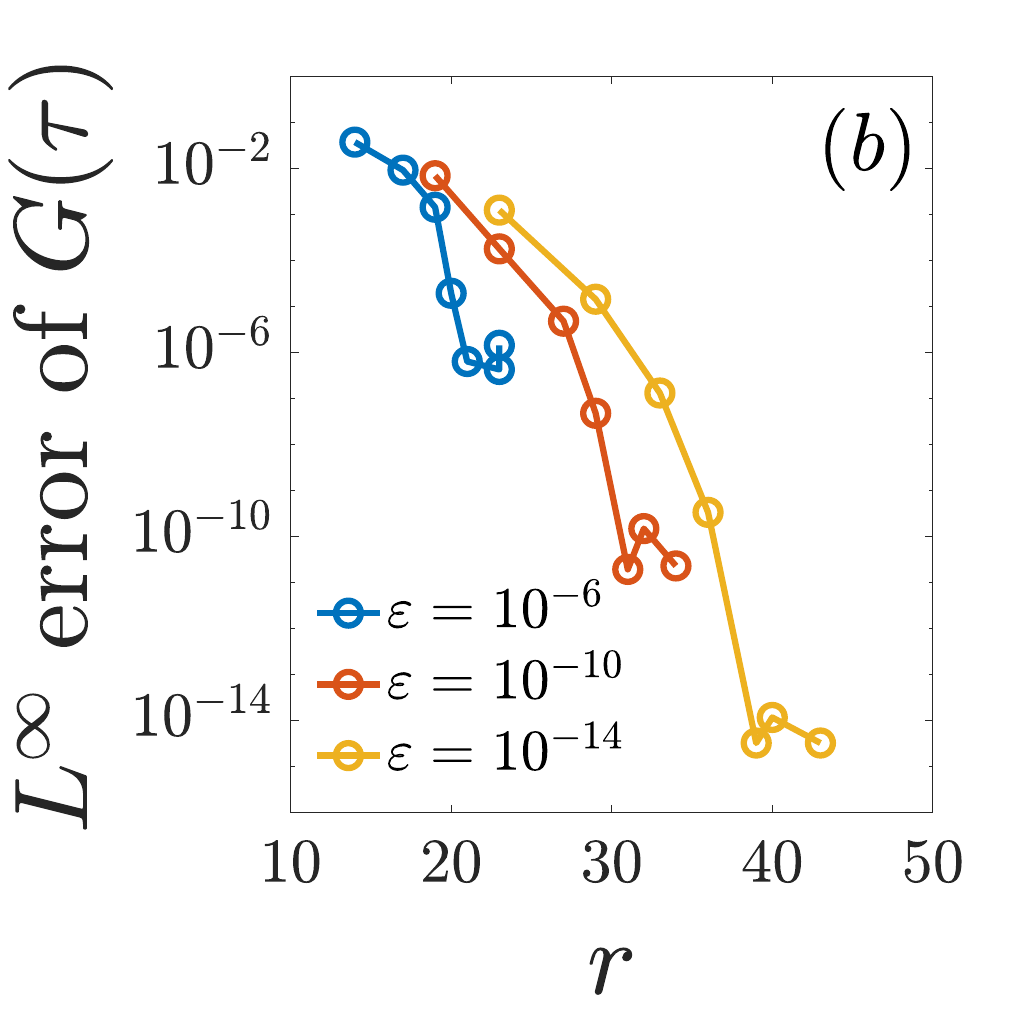}
  \includegraphics[width=.23\textwidth]{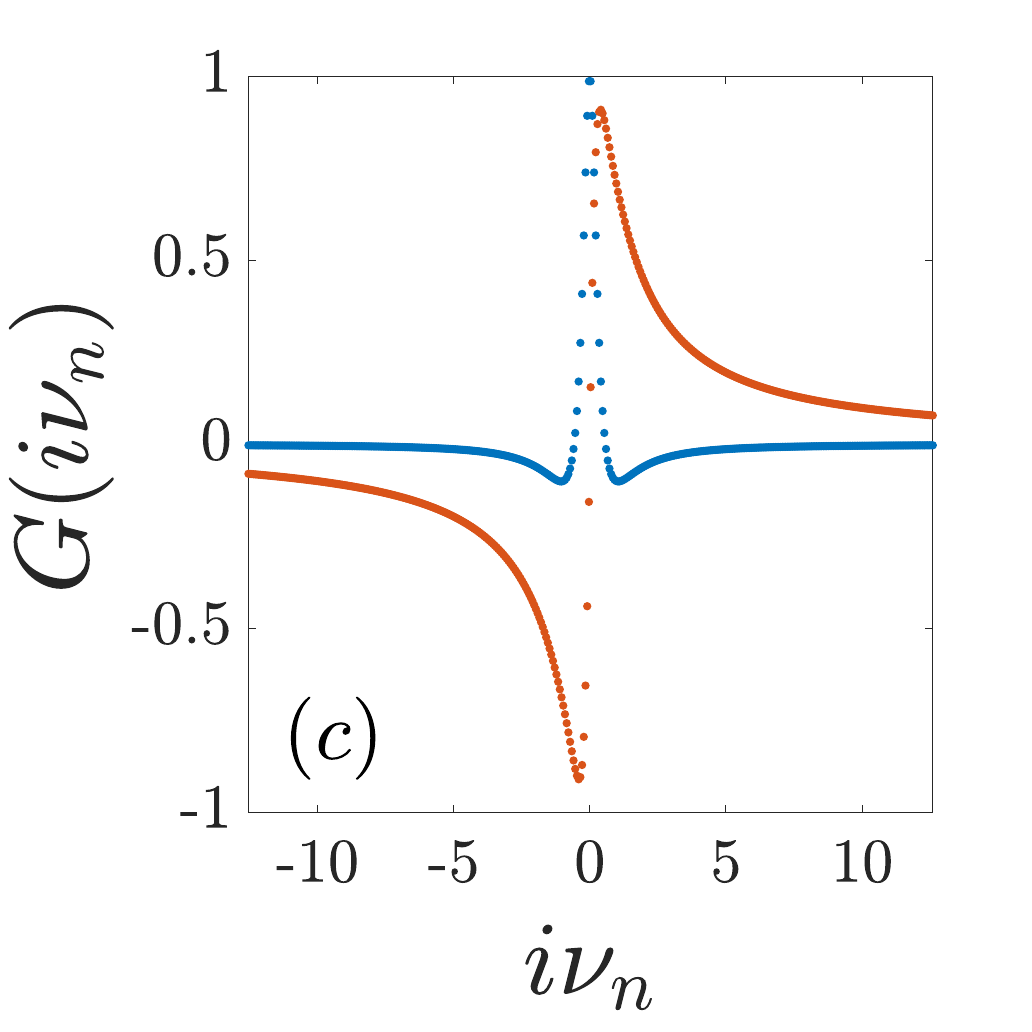}
  \includegraphics[width=.23\textwidth]{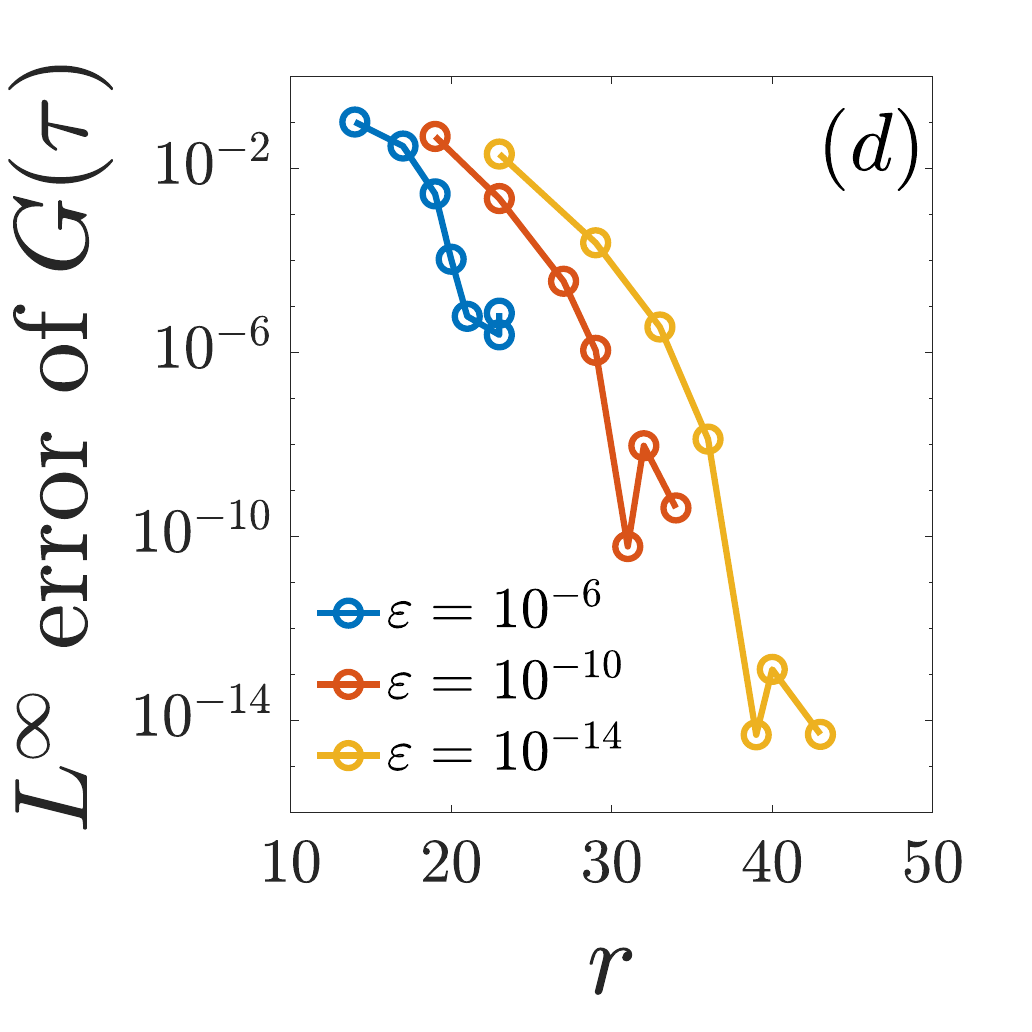}

  \caption{$L^\infty$ error of the DLR approximation of $G(\tau) =
  -\frac{1}{2} \int_{-1}^1 K(\tau,\omega) \paren{\delta(-1/3)+\delta(1)} \, d\omega$
  for $\beta = 100$ and several choices
  of $\epsilon$. (a) $G(\tau)$. (b) Error versus $r$ for imaginary time
  sampling. (c) $G(i\nu_n)$; real part in blue, and imaginary part in
  orange. (d) Error versus $r$ for Matsubara frequency sampling.}
  \label{fig:err_ha_1e2}
\end{figure*}

To test the Matsubara frequency sampling approach, we repeat the same
experiments, except that we recover the DLR
coefficients from samples of the Green's function on the Matsubara
frequency DLR grid. As before, we measure the error in
the imaginary time domain. Results for $\rho(\omega) =
\frac{2}{\pi} \sqrt{1-\omega^2} \theta\paren{1-\omega^2}$ with $\beta = 10^4$ are shown
in Figure \ref{fig:err_sc_mf_1e4}. These can compared with Figure
\ref{fig:semicirc}c. We
observe only a mild loss of accuracy compared with the
results obtained using imaginary time sampling, and we still achieve accuracy near $\epsilon$ when
$\Lambda$ is increased beyond the known cutoff. Results for
$\rho(\omega) = \paren{\delta(-1/3) + \delta(1)}/2$
with $\beta = 100$ are
shown in Figures \ref{fig:err_ha_1e2}c and \ref{fig:err_ha_1e2}d.
We have tested other choices of $\beta$ for both examples, up to $\beta
= 10^6$, with similar results.

\begin{figure*}[t]
  \centering
  \includegraphics[width=.23\textwidth]{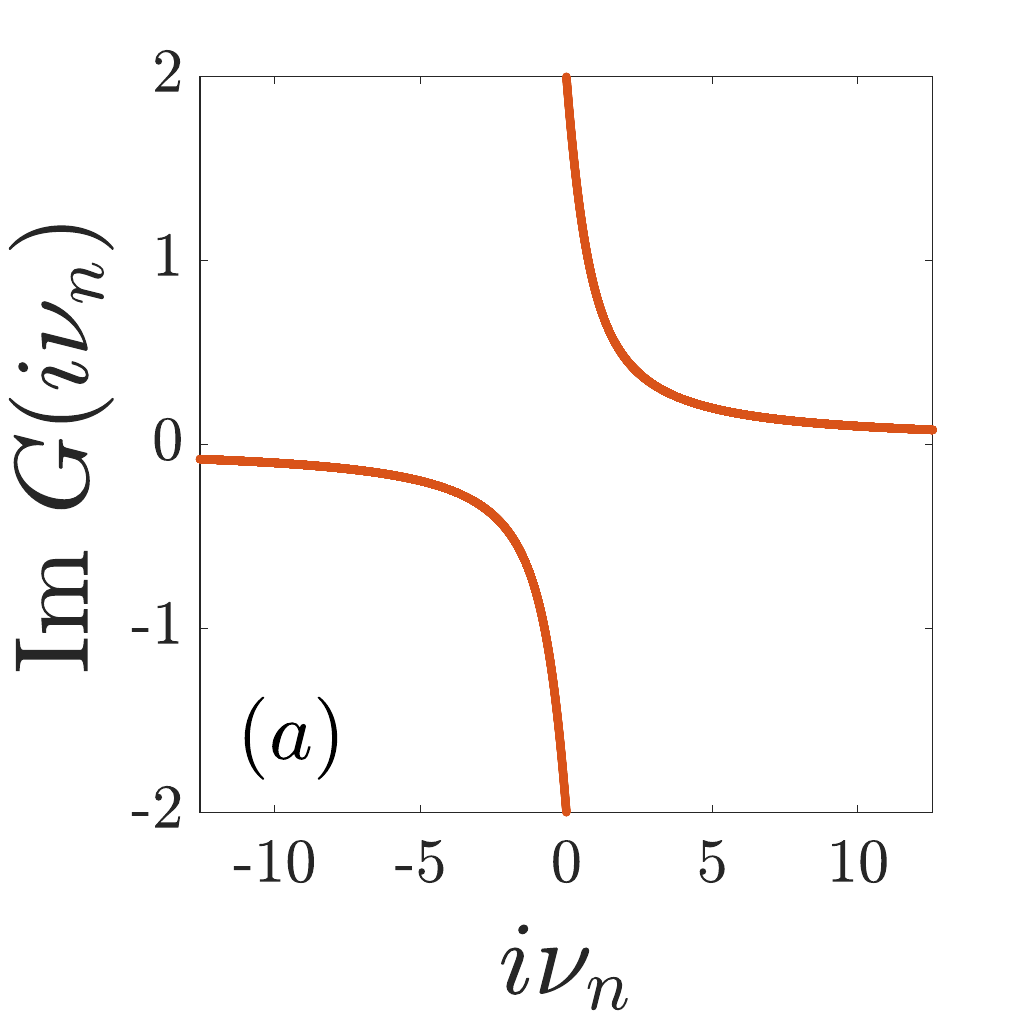}
  \includegraphics[width=.23\textwidth]{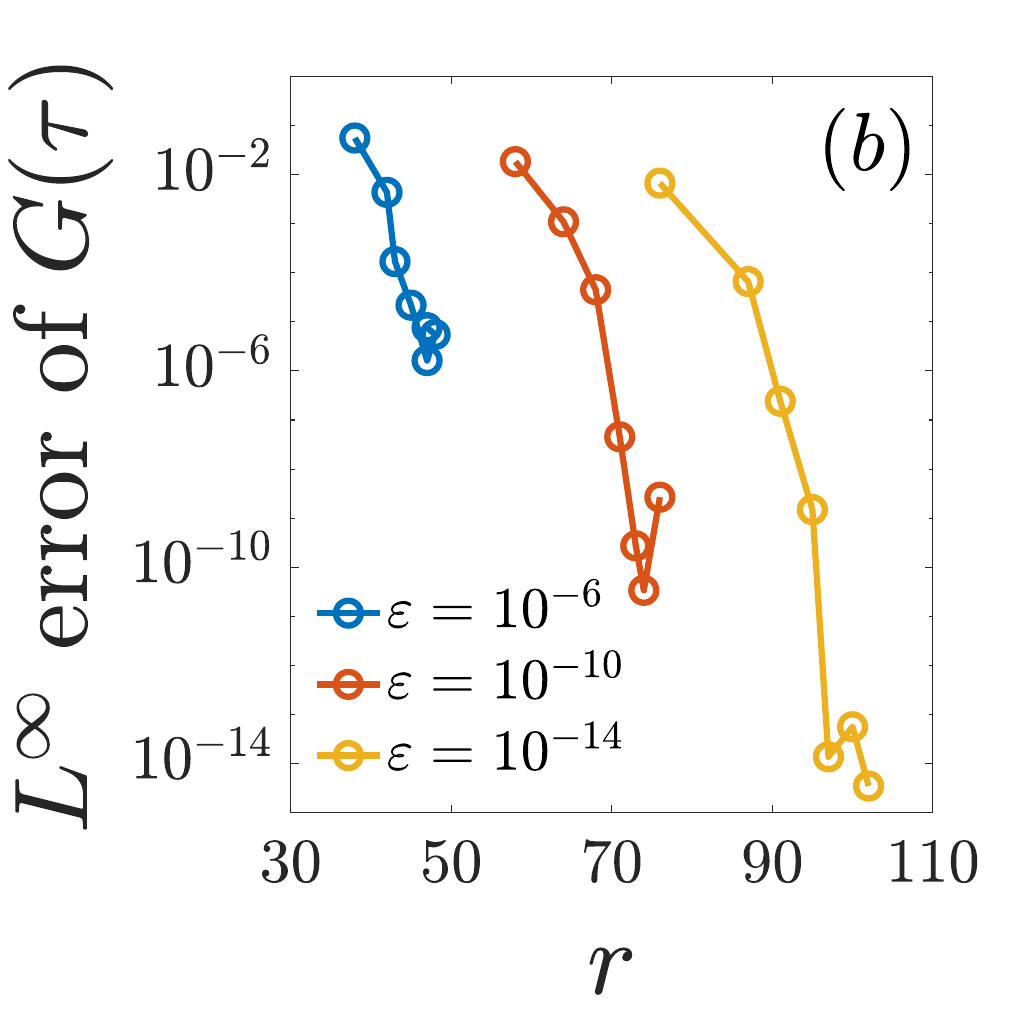}
  \includegraphics[width=.23\textwidth]{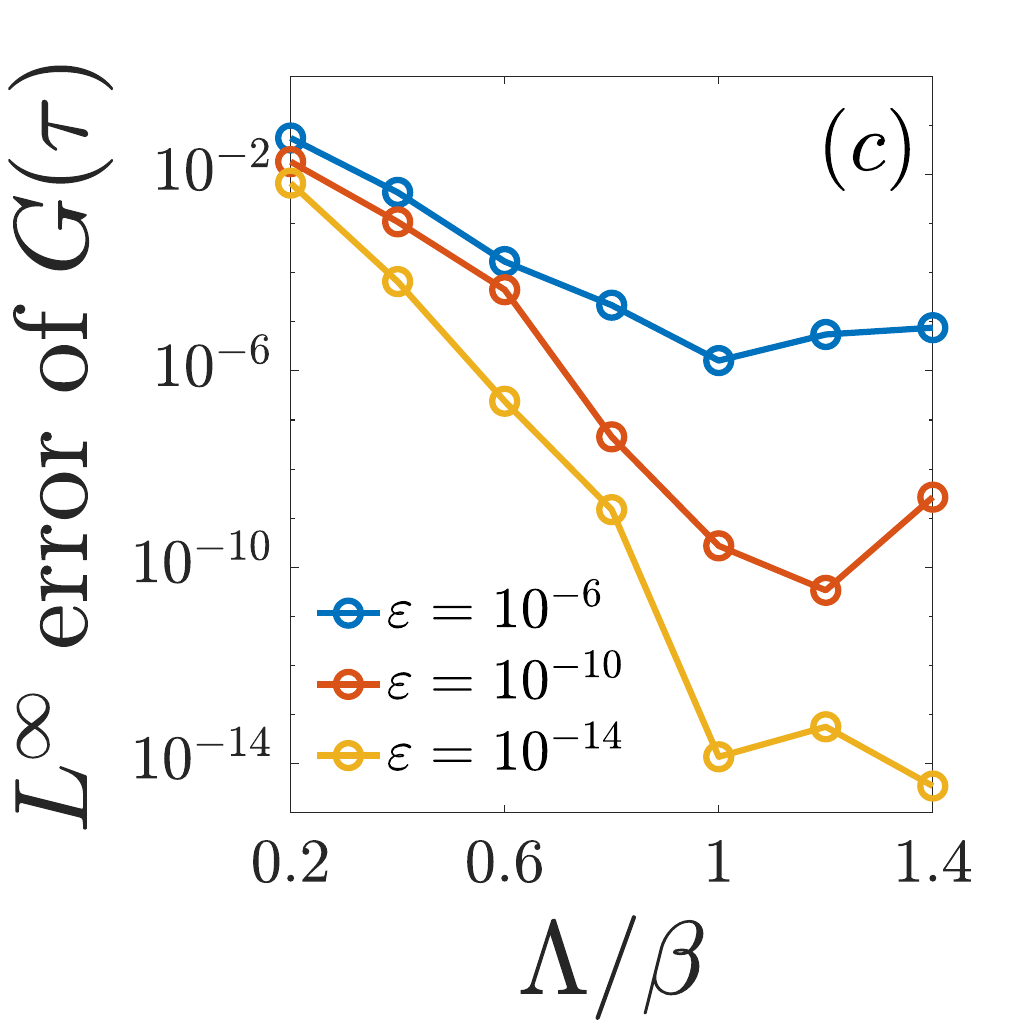}

  \caption{$L^\infty$ error of the DLR approximation of $G(\tau) =
  -\frac{2}{\pi} \int_{-1}^1 K(\tau,\omega) \sqrt{1-\omega^2} \, d\omega$ obtained
  using Matsubara frequency sampling for $\beta = 10^4$ and several choices
  of $\epsilon$. (a) $\Im G(i \nu_n)$; note that $\Re G(i
  \nu_n) = 0$ by symmetry. (b) Error versus $r$. (c) Error versus
  $\Lambda$.}
\label{fig:err_sc_mf_1e4}
\end{figure*}

\section{Intermediate representation} \label{sec:ir}

In this section, we rederive the intermediate representation (IR)
presented in Ref. \onlinecite{shinaoka17} 
using the tools we have introduced to construct the DLR. The IR uses an orthonormal basis
obtained from the SVD of an appropriate discretization of the kernel
$K(\tau,\omega)$.
It represents the same space as DLR, but has the advantage of orthogonality,
at the cost of using more complicated basis functions.
Our presentation of the IR differs from Refs. \onlinecite{shinaoka17, chikano18,chikano19,shinaoka21_2}
in two ways.

First, we show that discretizing $K$ on a composite grid like that
introduced in Section \ref{sec:kdisc} leads to an
efficient construction of the IR basis.
By contrast, in Ref. \onlinecite{chikano18}, an automatic adaptive
algorithm is used.
The authors report in Ref. \onlinecite{chikano19} that this algorithm takes on the order
of hours to build the IR basis for $\Lambda = 10^4$. To
address this problem, the library \texttt{irbasis} contains
precomputed basis functions for several values of $\Lambda$, and codes
to work with them. \cite{chikano19}
While this is a sufficient solution
for many cases, it may be restrictive in others, for example in
converging the IR with respect to $\Lambda$, or selecting $\Lambda$ to
achieve a given accuracy with the smallest possible number of basis
functions. 
Our approach, presented in Section \ref{sec:irbasis}, does not require
an expensive automatic adaptive algorithm. The IR basis is obtained by discretizing
$K$ on a well-chosen grid, as before, and computing a single SVD of a matrix whose dimension grows logarithmically with $\Lambda$, and
for $\Lambda$ up to $10^6$ is less than $1000$.
As an illustration, Figure \ref{fig:iranddlr} contains plots of a few IR and DLR basis functions for
$\Lambda = 10^4$ and $\epsilon = 10^{-14}$. Building each basis takes
less than a second, despite the high resolution required.

Second, we show in Section \ref{sec:irgrid} that the interpolative
decomposition of a matrix containing the $r$ IR basis functions naturally 
yields a set of $r$ sampling nodes for the IR, analogous to the
interpolation grid for the DLR, and a transformation from
values of a Green's function at these nodes to its IR coefficients.
In previous works, the sparse sampling method was used to provide
such a sampling grid for the IR. \cite{li20}
The sparse sampling nodes are chosen based on a heuristic, which is
motivated by the relationship between orthogonal polynomials and their
associated interpolation grids. While this heuristic appears to lead to a
numerically stable algorithm, the procedure we have used to construct the DLR
and Matsubara frequency grids is automatic and offers robust accuracy
guarantees. 

\begin{figure}[t]
  \centering
    \includegraphics[width=.23\textwidth]{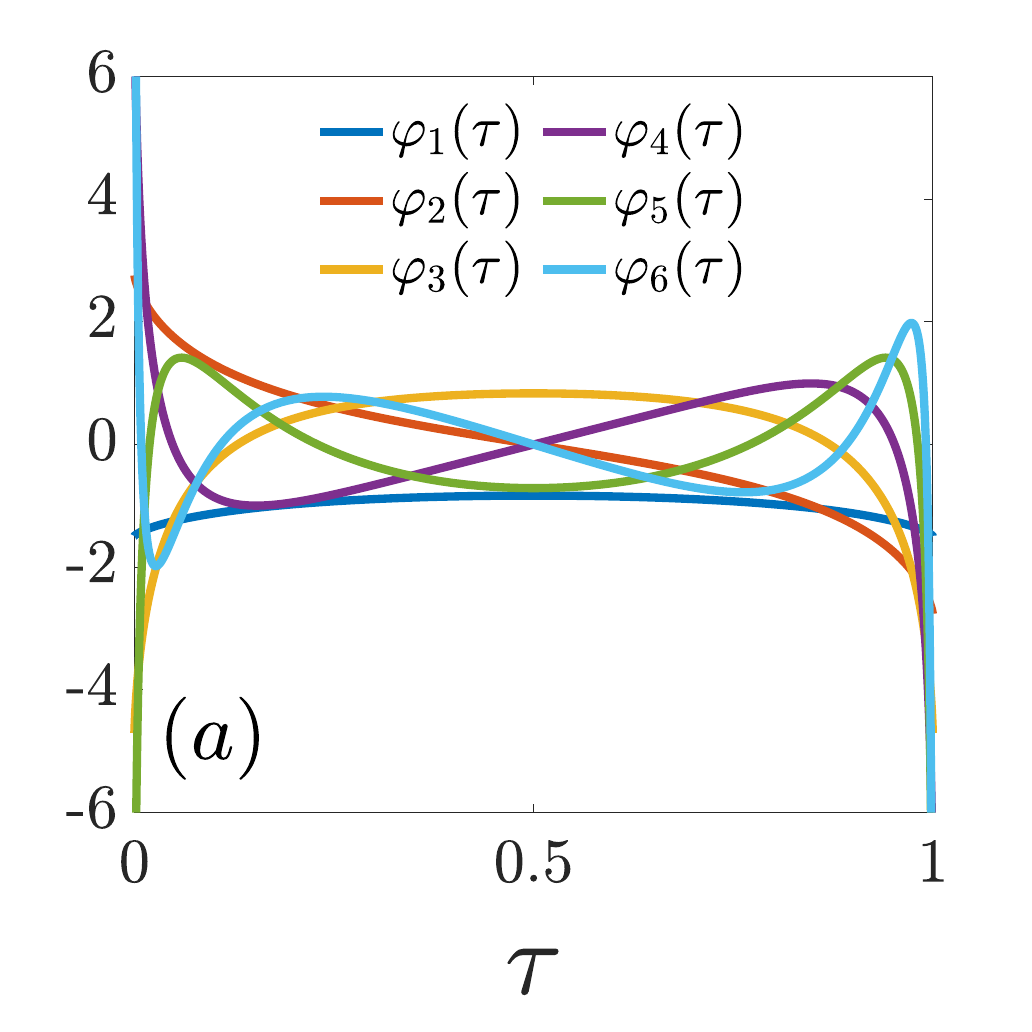}
    \includegraphics[width=.23\textwidth]{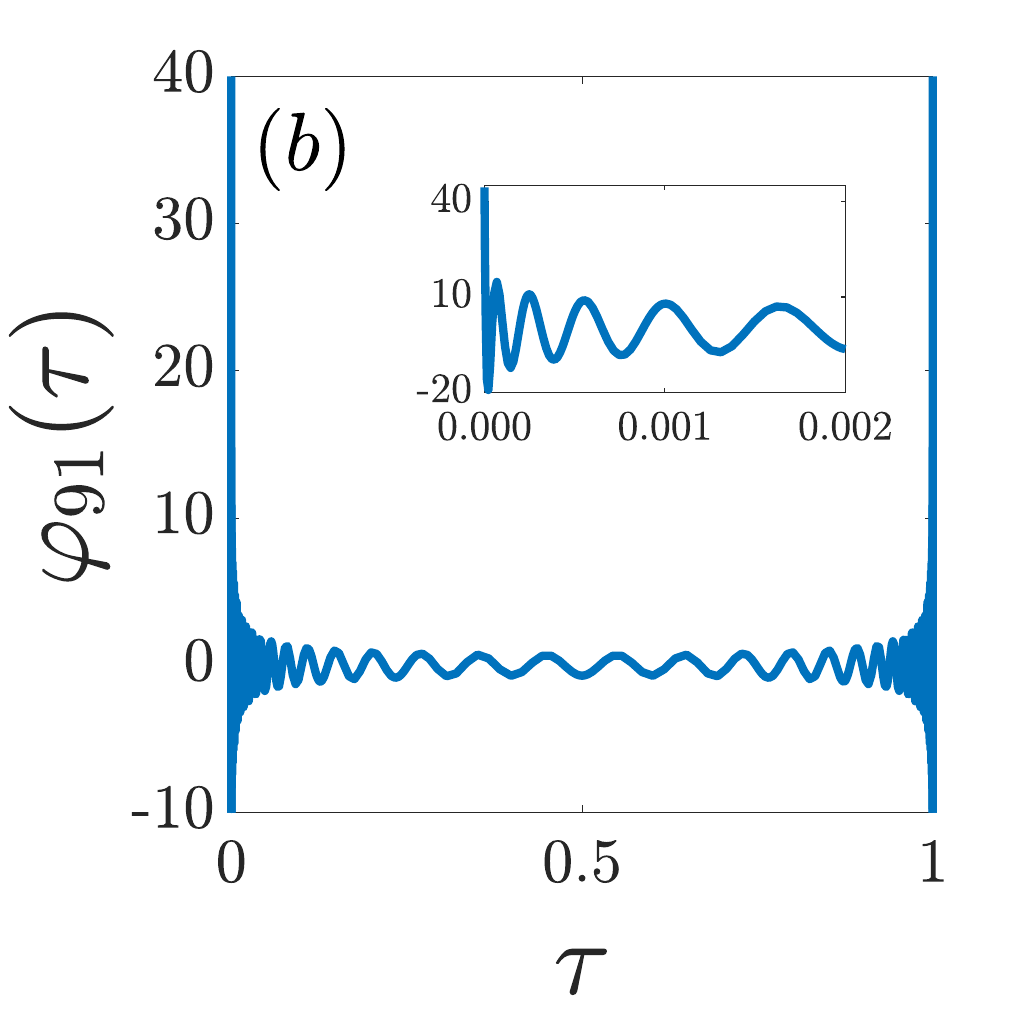}
    \includegraphics[width=.23\textwidth]{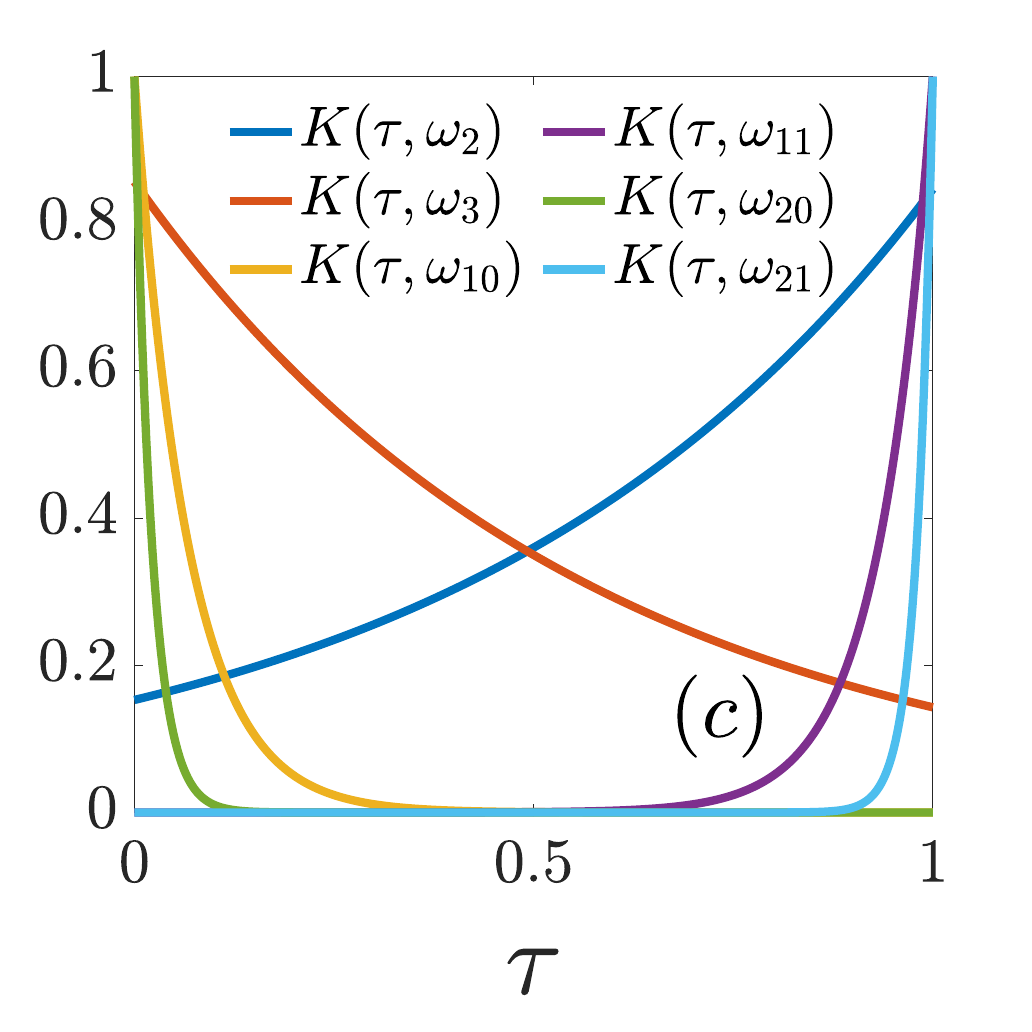}
    \includegraphics[width=.23\textwidth]{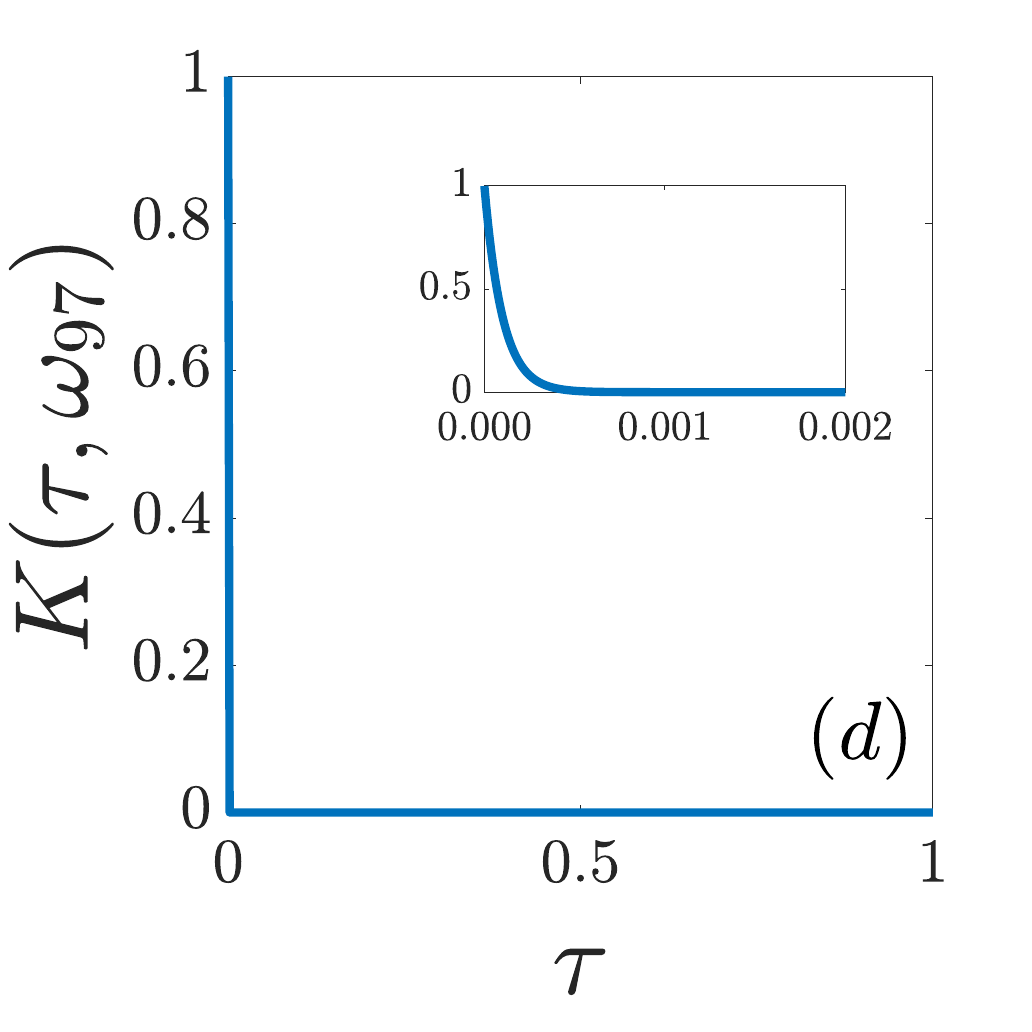}

  \caption{IR and DLR basis functions for $\Lambda = 10^4$ and
  $\epsilon = 10^{-14}$. (a) The first five IR basis functions. (b)
  The highest-degree (91st) IR basis function for the given parameters. (c) Several DLR
  basis functions for smaller $\omega_l$; we have ordered $\omega_l$ so
  that increasing $l$ corresponds to larger exponential rates. (d) The DLR
  basis function (96th) with the largest exponential decay rate for the given
  parameters.}
\label{fig:iranddlr}
\end{figure}

\subsection{The IR basis} \label{sec:irbasis}

The first step in constructing the IR basis is again to finely discretize
$K(\tau,\omega)$. Here, to ensure that we obtain a basis which is
orthogonal in the $L^2$ inner product, we use composite Legendre grids rather than
composite Chebyshev grids. The discussion in Section
\ref{sec:kdisc} holds equally well for composite Legendre grids, with
Gauss-Legendre nodes used in place of Chebyshev nodes.

In particular, let $\{\tau_i^f\}_{i=1}^M$ and $\{\omega_j^f\}_{j=1}^N$ be the nodes of the
composite Legendre fine grids in $\tau$ and $\omega$, respectively, and
let $A \in \RR^{M \times N}$ be the matrix with entries $A_{ij} =
K(\tau_i^f,\omega_j^f)$. Let $W \in \RR^{M \times M}$ be a diagonal
matrix with entries $W_{ii} = w_i^f$, the quadrature weights associated with the
composite Legendre grid points $\tau_i^f$. The quadrature weights
$\{w_i^f\}_{i=1}^M$ are obtained from the ordinary Gauss-Legendre quadrature
weights at $p$ Legendre nodes, rescaled to account for the
panel length.

Consider the SVD $\sqrt{W}A = U \Sigma V^T$ of
the reweighted matrix. Truncating the SVD at rank $r$ gives 
\[\sqrt{w_i^f} K(\tau_i^f,\omega_j^f) = \sum_{l=1}^r \sigma_l (u_l)_i
(v_l)_j + \errmat_{ij}\]
where $\sigma_l$, $\{u_l\}_{l=1}^r$, and $\{v_l\}_{l=1}^r$ are the first
$r$ singular values, left singular vectors, and right singular vectors,
respectively, and $\errmat$ is an error matrix. As before, we choose $r$ so that
$\norm{\errmat}_2 < \epsilon$, implying $r$ is the $\epsilon$-rank
of $\sqrt{W}A$.

Note that the entries of each left singular vector $u_l$ can be interpreted as
samples of a function on the fine grid in $\tau$, and similarly, the
entries of $v_l$ as samples of a function on the fine grid in $\omega$.
Summing against the corresponding truncated Lagrange
polynomials, we find
\begin{multline*}
K(\tau,\omega) = \sum_{l=1}^r \sigma_l \paren{\sum_{i=1}^M
  \wb{\ell}_i(\tau) \frac{(u_l)_i}{\sqrt{w_i^f}}} \paren{\sum_{j=1}^N
  \wb{\ell}_j(\omega) (v_l)_j} \\ + \errmat(\tau,\omega),
\end{multline*}
with $\errmat(\tau,\omega) = \sum_{i=1}^M \sum_{j=1}^N
\frac{\wb{\ell}_i(\tau)}{\sqrt{w_i^f}} \errmat_{ij}
\wb{\ell}_j(\omega)$.
Inserting this into the truncated Lehmann representation \eqref{eq:tlehmann}, we
obtain
\begin{multline*}
G(\tau) = -\sum_{l=1}^r \sigma_l \paren{\sum_{i=1}^M
  \wb{\ell}_i(\tau) \frac{(u_l)_i}{\sqrt{w_i^f}}} \\ \times
  \int_{-\Lambda}^\Lambda \paren{\sum_{j=1}^N
  \wb{\ell}_j(\omega) (v_l)_j} \rho(\omega) \, d\omega \\ -
  \int_{-\Lambda}^\Lambda \errmat(\tau,\omega) \rho(\omega) \, d\omega.
\end{multline*}
This establishes the validity of the representation
\[G(\tau) = \sum_{l=1}^r \wh{g}_l \varphi_l(\tau) + \errfun(\tau)\]
for
\[\varphi_l(\tau) = \sum_{i=1}^M \wb{\ell}_i(\tau)
\frac{(u_l)_i}{\sqrt{w_i^f}},\]
and $\errfun(\tau)$ an error term, 
analogous to the result in Theorem \ref{thm:dlr}. We do not give an
explicit bound on $\errfun(\tau)$ here, but evidently it is similar to
that for the DLR case.

The orthonormality of the collection
$\{\varphi_l\}_{l=1}^r$ follows from that of the left singular vectors
$\{u_l\}_{l=1}^r$:
\begin{align} \label{eq:phiorth}
  \begin{split}
    \int_0^1 \varphi_k(\tau) \varphi_l(\tau) \, d\tau &= \sum_{i=1}^M
\varphi_k(\tau_i^f) \varphi_l(\tau_i^f) w_i^f \\
    &= \sum_{i=1}^M (u_k)_i (u_l)_i = \delta_{kl}.
  \end{split}
\end{align}
Here, the first equality holds because the functions
$\{\varphi_l\}_{l=1}^r$ are piecewise polynomials of degree $p-1$, so the
Gauss-Legendre quadrature rule is exact, and the second follows from the
definition of $\varphi_l$ and the truncated Lagrange
polynomials. We define the IR basis as
$\{\varphi_l\}_{l=1}^r$.

The functions $\varphi_l$ are represented using the singular
vectors $\{u_l\}_{l=1}^r$ of $\sqrt{W}A$, so
constructing them only requires forming and computing the SVD of this $M
\times N$ matrix, with $M, N = \OO{\log \Lambda}$, truncated to include
only singular values larger than some desired accuracy $\epsilon$.

Operations involving the IR basis functions are straightforwardly
carried out by working with the piecewise polynomial representation. For
example, to evaluate $\varphi_l$ at a point $\tau$, we first find the
subinterval in the composite Legendre grid
containing $\tau$, and then evaluate a Legendre expansion on that
subinterval at $\tau$.
It follows from the orthonormality of the IR basis, and the
exactness of Gauss-Legendre quadrature on polynomials of degree $2p-1$, that the
IR coefficients of a Green's function
\begin{equation} \label{eq:gir}
  G(\tau) = \sum_{l=1}^r \wh{g}_l \varphi_l(\tau)
\end{equation}
are given by
\begin{align} \label{eq:ircoefs}
  \begin{split}
    \wh{g}_l &= \int_0^1 \varphi_l(\tau) G(\tau) \, d\tau \\ 
    &= \sum_{i=1}^M \varphi_l(\tau_i^f) G(\tau_i^f) w_i^f = \sum_{i=1}^M (u_l)_i \,
  G(\tau_i^f) \, \sqrt{w_i^f}.
  \end{split}
\end{align}

\subsection{The imaginary time IR grid and transform matrix} \label{sec:irgrid}

Computing the IR coefficients using \eqref{eq:ircoefs} requires sampling
$G(\tau)$ at $M \gg r$ grid points. As for the imaginary time DLR grid, we show how to obtain $r$
\emph{imaginary time IR grid points} $\{\tau_i\}_{i=1}^r$ and an $r \times r$ transform matrix $T$ so that given a Green's function \eqref{eq:gir},
we have $\wh{g}_l \approx \sum_{k=1}^r T_{lk} \, G(\tau_k)$ to high
accuracy. We note that
since the IR basis is orthogonal, it is natural to use projection rather
than interpolation to obtain the expansion coefficients, so the
procedure here is different than that for the DLR basis.

Let $\Phi$ be the matrix containing the IR basis functions on the fine
grid, $\Phi_{ij} = \varphi_j(\tau_i^f) =
(u_j)_i/\sqrt{w_i^f}$.
The ID of $\Phi^T$ gives
\[\Phi = R \phi\]
with $\phi \in \RR^{r \times r}$ consisting of selected rows of $\Phi$, and
$R \in \RR^{M \times r}$ the projection matrix. We take
$\{\tau_k\}_{k=1}^r$ to be the subset of the fine grid points
$\{\tau_i^f\}_{i=1}^M$ corresponding to the selected rows of $\Phi$, and
define an $r \times r$ matrix
\begin{equation} \label{eq:idmat}
  T = \Phi^T W R.
\end{equation}

Suppose $G$ is given by \eqref{eq:gir}, and let $g, \wh{g} \in \RR^r$
with $g_k = G(\tau_k)$. In particular, we have $\phi \wh{g} = g$. Then
\[Tg = \Phi^T W R g = \Phi^T W R \phi \wh{g} = \Phi^T W \Phi \wh{g} =
\wh{g}\]
since $\Phi^T W \Phi = I$ from \eqref{eq:phiorth}.
Thus, the imaginary time IR grid points and transform matrix can be computed directly
from the ID of $\Phi$, and can be used to recover the IR coefficients
from the values of a Green's function on the IR grid.

We note that since the IR basis is orthogonal, issues of stability are more
straightforward than in the DLR case, and we do not give a detailed
analysis here.

\subsection{IR in the Matsubara frequency domain}

One can construct a Matsubara frequency grid for the IR basis using
similar techniques to those presented in Section \ref{sec:matpts}. In this case, however, we do not have
simple analytical expressions for the Fourier transforms of the IR basis
functions, and these have to be computed by numerical integration
using the piecewise polynomial
representations. This process is cumbersome compared with the
analogous method for the DLR basis, and we will not describe it in detail.

As an alternative, to recover the IR coefficients from samples of a
Green's function in the Matsubara frequency domain, one could
simply evaluate the Green's function on the Matsubara frequency DLR
grid, recover the DLR coefficients, evaluate the resulting DLR expansion
on the IR grid, and apply the transform $T$.

\section{Dyson equation in the DLR basis} \label{sec:dyson}

We consider the Dyson equation relating a Matsubara Green's function
and self-energy,
\begin{equation} \label{eq:dysonmat}
   G^{-1}(i\nu_n) = G^{-1}_0(i\nu_n) - \Sigma(i\nu_n), 
\end{equation}
where $G_0$ is a given Matsubara Green's function.
Although it is diagonal in the Matsubara frequency domain, it can also
be written in the time domain as an integral equation,
\begin{equation} \label{eq:dyson}
    G(\tau) - \int_0^\beta d\tau' G_0(\tau-\tau') \int_0^\beta d\tau'' 
    \Sigma(\tau'-\tau'') G(\tau'') = G_0(\tau).
\end{equation}
The functions $G$, $G_0$, and $\Sigma$ can be extended to $(-\beta,0)$ using
the $\beta$-antiperiodicity property $f(-\tau) = -f(\beta-\tau)$ or the
$\beta$-periodicity property
$f(-\tau) = f(\beta-\tau)$ for fermionic and bosonic Green's functions,
respectively. Since $G(\tau)$ is an imaginary
time Green's function, it has a Lehmann spectral representation
\eqref{eq:lehmann}, and can therefore be approximated by a DLR. We
assume the same is true of the self-energy $\Sigma$, and of the
intermediate convolutions in \eqref{eq:dyson}; this can be shown in many
typical cases of physical interest. For simplicity, we assume in this section that all
quantities are fermionic, but our discussion is straightforwardly
extended to the bosonic case.

Since $\Sigma$ in general depends on $G$, the Dyson
equation must be solved self-consistently by nonlinear iteration: see
for example \eqref{eq:syk} in the next section for the SYK self-energy. The
standard method is to compute $\Sigma$ in the imaginary time
domain, where it is typically simpler, and to solve the Dyson equation
\eqref{eq:dysonmat} in the Matsubara frequency domain where it is
diagonal. This
procedure can be carried out efficiently using the DLR: (i) given
$G$ on the imaginary time DLR grid computed from a previous iterate,
$\Sigma$ is computed on the imaginary time DLR grid; (ii) the DLR
coefficients of $\Sigma$ are recovered; (iii) $\Sigma$ is evaluated
on the Matsubara frequency grid; (iv) \eqref{eq:dysonmat} is solved to
obtain $G$ on the Matsubara frequency grid; (v) the DLR coefficients of
$G$ are recovered; and (vi) $G$ is evaluated on the imaginary time DLR
grid to prepare for the next iterate.  Ref. \onlinecite{li20} describes
and demonstrates a similar procedure using the sparse sampling method
for the IR.

In this section, we show how to solve the Dyson equation directly in
imaginary time using the DLR basis. We note that much of the discussion
holds equally well for the IR basis -- or any other basis, including an
orthogonal polynomial basis \cite{GullStrand_2020} -- however, certain
quantities which must be computed by numerical integration in that case
are given analytically for the DLR basis. We will work with the integral
form \eqref{eq:dyson}, and assume $\Sigma$ is given, as is the case
within a single step of nonlinear iteration.

Let $G$ be a Green's function given by a DLR
\[G(\tau) = \sum_{l=1}^r K(\tau,\omega_l) \wh{g}_l\]
and let $g_k = G(\tau_k)$.
We will use similar notation for other quantities.
We define the convolution between $\Sigma$ and $G$ by
\begin{equation} \label{eq:fconv}
  F(\tau) \equiv \int_0^1 \Sigma(\tau-\tau') G(\tau') \, d\tau'.
\end{equation}
Let $\wb{\Sigma} \in \RR^{r \times r}$ denote the matrix discretizing this convolution, so that
\begin{equation} \label{eq:sigconv}
  f = \wb{\Sigma} g
\end{equation}
with $f_k = F(\tau_k)$. $\wb{\Sigma}$ can be constructed by a linear
transformation of the values $\sigma_k = \Sigma(\tau_k)$; there is a
tensor $\mathcal{C}_{ijk}$ with
\begin{equation} \label{eq:getsig1}
  \wb{\Sigma}_{ij} = \sum_{k=1}^r \mathcal{C}_{ijk} \sigma_k.
\end{equation}
As we will see, it may be simpler to form $\wb{\Sigma}$ from its DLR
coefficients $\wh{\sigma}_l$, and there is a tensor
$\wh{\mathcal{C}}_{ijl}$ with
\begin{equation} \label{eq:getsig2}
  \wb{\Sigma}_{ij} = \sum_{l=1}^r \wh{\mathcal{C}}_{ijl} \wh{\sigma}_l.
\end{equation}

Using this notation, the discretization of \eqref{eq:dyson} in the DLR
basis is given by
\begin{equation} \label{eq:dysonDLR}
  (I - \wb{G}_0 \wb{\Sigma}) g = g_0,
\end{equation}
where $\wb{G}_0$ can be obtained as in \eqref{eq:getsig1} or
\eqref{eq:getsig2}. This is simply an $r \times r$ linear system. Thus,
given $\Sigma$, $\wb{\Sigma}$ can be obtained using \eqref{eq:getsig1}
or \eqref{eq:getsig2}, and then \eqref{eq:dysonDLR} can be solved to obtain
$G$ on the imaginary time DLR grid.
It remains only to discuss the construction of the tensors $\mathcal{C}$ and $\wh{\mathcal{C}}$.

We begin by discretizing the convolution \eqref{eq:fconv} on the
imaginary time DLR grid:
\begin{align*}
  f_k &= F(\tau_k) = \int_0^1 \Sigma(\tau_k-\tau') G(\tau') d\tau' \\
  &= \sum_{l=1}^r \paren{\int_0^1 \Sigma(\tau_k-\tau')
  K(\tau',\omega_l) \, d\tau'} \wh{g}_l \equiv \sum_{l=1}^r
  \wh{\Sigma}_{kl} \wh{g}_l.
\end{align*}
Here we have defined $\wh{\Sigma}$ as
the matrix of convolution by $\Sigma$, which takes the DLR coefficients
$\wh{g}_l$ to the values $f_l$ of the convolution at the imaginary time
DLR grid points. Recall the matrix $\kmat$ defined by
\eqref{eq:kmatdef}, which gives $g = \kmat \wh{g}$.
Precomposing $\wh{\Sigma}$ with $\kmat^{-1}$, we obtain the matrix
\[\wb{\Sigma} = \wh{\Sigma} \kmat^{-1}\]
yielding \eqref{eq:sigconv}.
We can define the matrix $\wb{G}_0$ of
convolution by $G_0$ similarly.

To construct $\wh{\Sigma}$, we take $\Sigma(\tau) = \sum_{k=1}^r
K(\tau,\omega_j) \wh{\sigma}_k$ and write
\begin{equation} \label{eq:formS}
  \begin{aligned}
  \wh{\Sigma}_{ij} &= \int_0^1 \Sigma(\tau_i-\tau')
    K(\tau',\omega_j) \, d\tau' \\
  &= \begin{multlined}[t] \int_0^{\tau_i} \Sigma(\tau_i-\tau') K(\tau',\omega_j) \, d\tau' \\ -
    \int_{\tau_i}^1 \Sigma(1+\tau_i-\tau') K(\tau',\omega_j) \, d\tau'
  \end{multlined} \\
    &= \begin{multlined}[t] \sum_{k=1}^r \wh{\sigma}_k \left(\int_0^{\tau_i}
    K(\tau_i-\tau',\omega_k) K(\tau',\omega_j) \,  d\tau'\right. \\
    \left. - \int_{\tau_i}^1 K(1+\tau_i-\tau',\omega_k)
    K(\tau',\omega_j) \, d\tau'\right)
    \end{multlined} \\
    &= \sum_{k=1}^r
      \wt{\mathcal{C}}_{ijk} \wh{\sigma}_k,
  \end{aligned}
\end{equation}
where we have used the antiperiodicity property.
A straightforward calculation shows that $\wt{\mathcal{C}}_{ijk}$ is given explicitly by 
\[\wt{\mathcal{C}}_{ijk} =
\begin{cases}
  \frac{K(\tau_i,\omega_j) - K(\tau_i,\omega_k)}{\omega_k-\omega_j}
  &\text{if } j \neq k \\
  \paren{\tau_i-K(1,\omega_j)} K(\tau_i,\omega_j) &\text{if } j = k.
\end{cases}
\]

The matrix $\wb{\Sigma}$
is then given by
\[\wb{\Sigma}_{ij} = \sum_{k=1}^r \wh{\Sigma}_{ik} \kmat^{-1}_{kj} =
\sum_{k,l=1}^r \wt{\mathcal{C}}_{ikl} \wh{\sigma}_l \kmat^{-1}_{kj}.\]
Defining
\begin{equation} \label{eq:Cdef}
  \wh{\mathcal{C}}_{ijl} \equiv \sum_{k=1}^r \wt{\mathcal{C}}_{ikl}
  \kmat^{-1}_{kj}
\end{equation}
gives \eqref{eq:getsig2}.
We remark that in practice $\kmat^{-1}$ should be applied in a numerically stable
manner, such as by LU factorization and back substitution, rather than
formed explicitly.

Inserting $\wh{\sigma} = \kmat^{-1}
\sigma$ into \eqref{eq:getsig2}, we obtain \eqref{eq:getsig1} with
\begin{equation} \label{eq:Chatdef}
  \mathcal{C}_{ijk} \equiv \sum_{l=1}^r \wh{\mathcal{C}}_{ijl}
  \kmat^{-1}_{lk}.
\end{equation}
However, if this computation is not done carefully, rounding error will
lead to a significant loss of precision. In order to maintain full
double precision accuracy using \eqref{eq:getsig1}, $\wt{\mathcal{C}}$ and
$\kmat$ must be formed in quadruple precision. This is of course straightforward, since
the entries of these arrays are given explicitly. Then,
\eqref{eq:Cdef} and \eqref{eq:Chatdef} must be computed in quadruple
precision. Once $\mathcal{C}$ has been obtained, all subsequent
calculations -- in particular, \eqref{eq:getsig1} -- can be
carried out in double precision. Describing this phenomenon requires an analysis of floating
point errors which is beyond this scope of this paper. Alternatively, one
can simply obtain $\wh{\sigma}$ from $\sigma$ first, and use
\eqref{eq:getsig2} instead of \eqref{eq:getsig1}; then no such issue
arises, and all arrays may be formed using double precision arithmetic.

We make a brief remark on the computational complexity of solving the
Dyson equation using the DLR. The more standard method, using
\eqref{eq:dysonmat}, scales as $\OO{r^2}$, due to the cost of
transforming between the imaginary time and Matsubara frequency DLR grid
representations of $G$ and $\Sigma$. The sparse sampling method is
similar, and has roughly the same cost. \cite{li20} By contrast, the
imaginary time domain method we have described scales as
$\OO{r^3}$, due to the cost of forming $\wb{\Sigma}$ (the system
\eqref{eq:dysonDLR} can typically be solved at an $\OO{r^2}$ cost using
an iterative linear solver). Methods of reducing this cost may exist,
and will be explored in the future. However, since $r$ is typically small, the
discrepancy may or may not be significant in practice, and the pure
imaginary time domain method may be more convenient or robust in certain
applications.

\section{Example: the SYK equation} \label{sec:syk}

To demonstrate the method described in the previous section, we consider the  
Sachdev-Ye-Kitaev (SYK) equations, given by \cite{SachdevYe93,gu20}
\begin{equation}\label{eq:syk}
  \begin{cases}
    G^{-1}(i\nu_n) = i\nu_n + \mu - \Sigma(i\nu_n) \\ 
    \Sigma(\tau) = J^2 G^2(\tau) G(\beta-\tau),
  \end{cases}
\end{equation}
where $\mu$ is the chemical potential, $J$ is a coupling constant, and $G$ is a fermionic Matsubara
Green's function. We fix $J = 1$.

The SYK model exhibits remarkable properties, and is the subject of a
large literature. \cite{chowdhury21} Here, our motivation
is to illustrate the efficiency of the DLR approach in
solving a nonlinear Dyson equation. In the $\beta \to
\infty$ limit, it is known that solutions develop a $1/\sqrt{\omega}$
non-Fermi liquid singularity at low frequencies, or equivalently
$1/\sqrt{\tau}$ decay at large imaginary times. \cite{SachdevYe93} The DLR expansion
captures this behavior with excellent accuracy. Although guaranteed by our
analysis, this result may appear counterintuitive, but there is in fact a
significant literature on the approximation of functions
with power law decay by sums of a small number of exponentials.
\cite{beylkin05,beylkin10,zhang21}

We solve (\ref{eq:syk}) in the DLR basis using the imaginary time domain
method described in Section \ref{sec:dyson}.
Nonlinear iteration is carried out using a weighted fixed
point iteration
\[\Sigma^{(n+1)} = \Sigma[w \, G^{(n)} + (1-w) \, G^{(n-1)}],\]
with the weight $w$ chosen to ensure convergence. We terminate the
iteration when the
values of $G^{(n)}$ and $G^{(n-1)}$ on the imaginary time DLR grid match
pointwise to within a fixed point tolerance $\epsilon_{\text{fp}}$.

We first solve (\ref{eq:syk}) with $\mu = 0$ and $\beta = 10^4$, using
$G(\tau) = -1/2$ as the initial guess for the
weighted fixed point iteration. We take $\epsilon = 10^{-14}$, $\Lambda
= 5\beta$,
$\epsilon_{\textrm{fp}} = 10^{-12}$, and $w = 0.15$.
The calculation involves systems of only $117$ degrees of freedom, and takes less than a
second on a laptop.
$G(\tau)$ is plotted in Figure \ref{fig:gsyk}a,
along with the conformal asymptotic solution $G_c(\tau)$ given by \cite{PhysRevB.63.134406, gu20}
\begin{equation} \label{eq:gconf}
   G_c(\tau) = - \frac{\pi^{1/4}}{\sqrt{2\beta}} \paren{\sin\paren{\frac{\pi
	    \tau}{\beta}}}^{-1/2}.
\end{equation}
In Figure \ref{fig:gsyk}b, we plot the difference $G(\tau)-G_c(\tau)$ for $\tau \in
[0,\beta/2]$. We observe the expected $\OO{\tau^{-3/2}}$
asymptotic correction to \eqref{eq:gconf}. In Figure \ref{fig:gsyk}c, we
plot the error of $G(\tau)$ as compared with a standard
Legendre polynomial-based solver, \cite{GullStrand_2020}
which operates according to the description in Section \ref{sec:dyson}
with the DLR basis and nodes replaced by a Legendre polynomial basis and
Legendre nodes.

\begin{figure}[t]
  \centering
  \includegraphics[width=\linewidth]{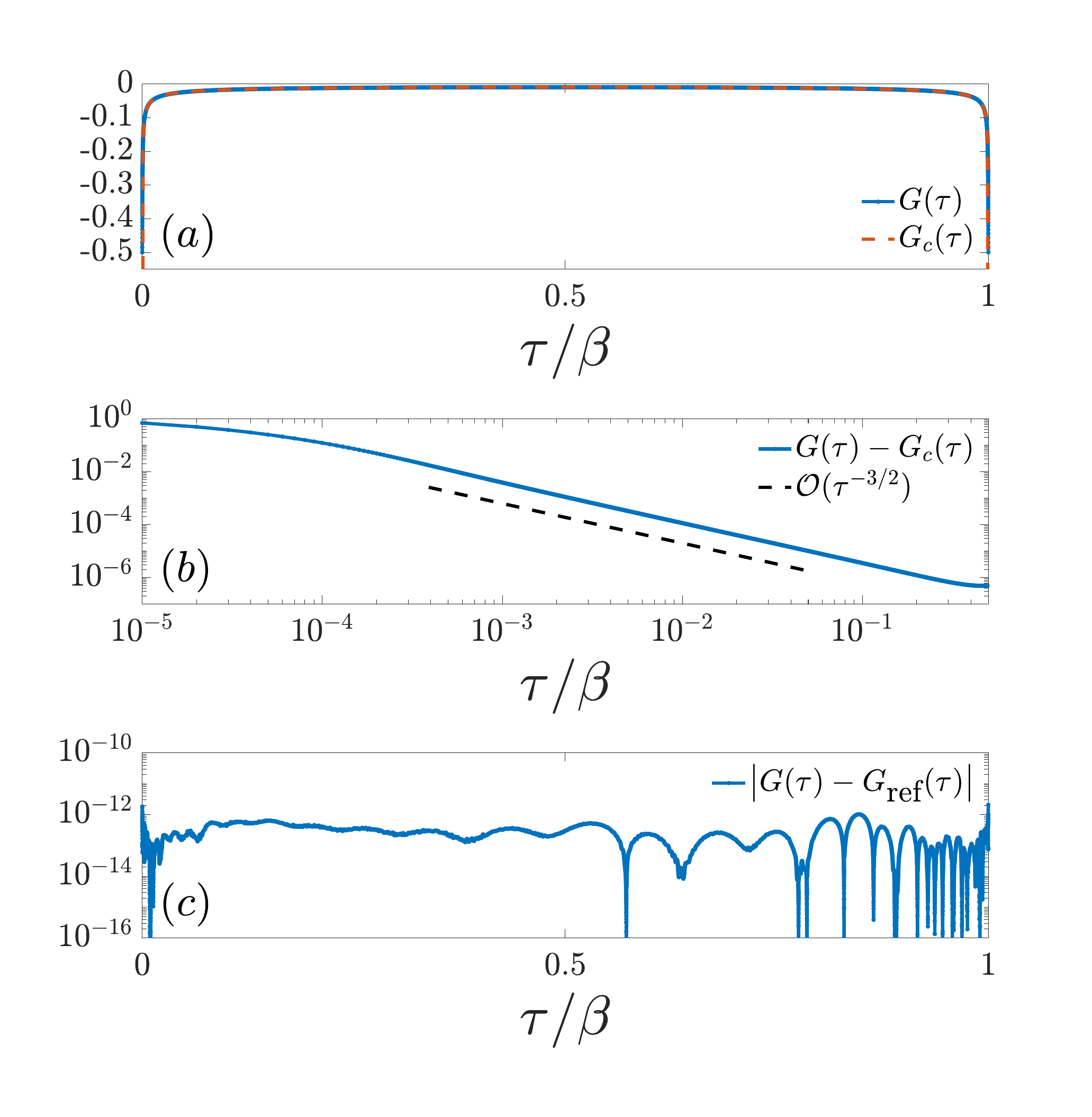}
  \caption{(a) Solution $G(\tau)$ of the SYK equation with $J = 1$, $\mu = 0$,
  and $\beta = 10^4$, along with the conformal solution $G_c(\tau)$.
  (b) Pointwise difference $G - G_c$, showing the form of the
  higher-order correction. (c) Pointwise error of computed $G$ measured
  against a reference solution $G_\text{ref}$ obtained using Legendre
  polynomial-based solver.}
\label{fig:gsyk}
\end{figure}

We next carry out
a high precision calculation of the compressibility in the SYK model in
the zero temperature limit,
following the results of Ref. \onlinecite{gu20} (Sec. 4.2).
We define the charge $Q$ (conventionally vanishing at half-filling) as 
\begin{equation}
   Q(\beta,\mu) \equiv (G_{\beta,\mu}(\beta)-G_{\beta,\mu}(0))/2,
\end{equation}
where $G_{\beta,\mu}$ is the solution of (\ref{eq:syk}) for fixed $\beta$ and $\mu > 0$.
The compressibility $K$ is defined as 
\begin{equation}
   K(T) = \left. \pd{Q(\beta,\mu)}{\mu}\right\rvert_{\mu = 0^+} =
\lim_{\mu \to 0^+} \frac{Q(\beta,\mu)}{\mu}
\end{equation}
with $T = \beta^{-1}$.

\begin{figure}[t]
  \centering
  \includegraphics[width=\linewidth]{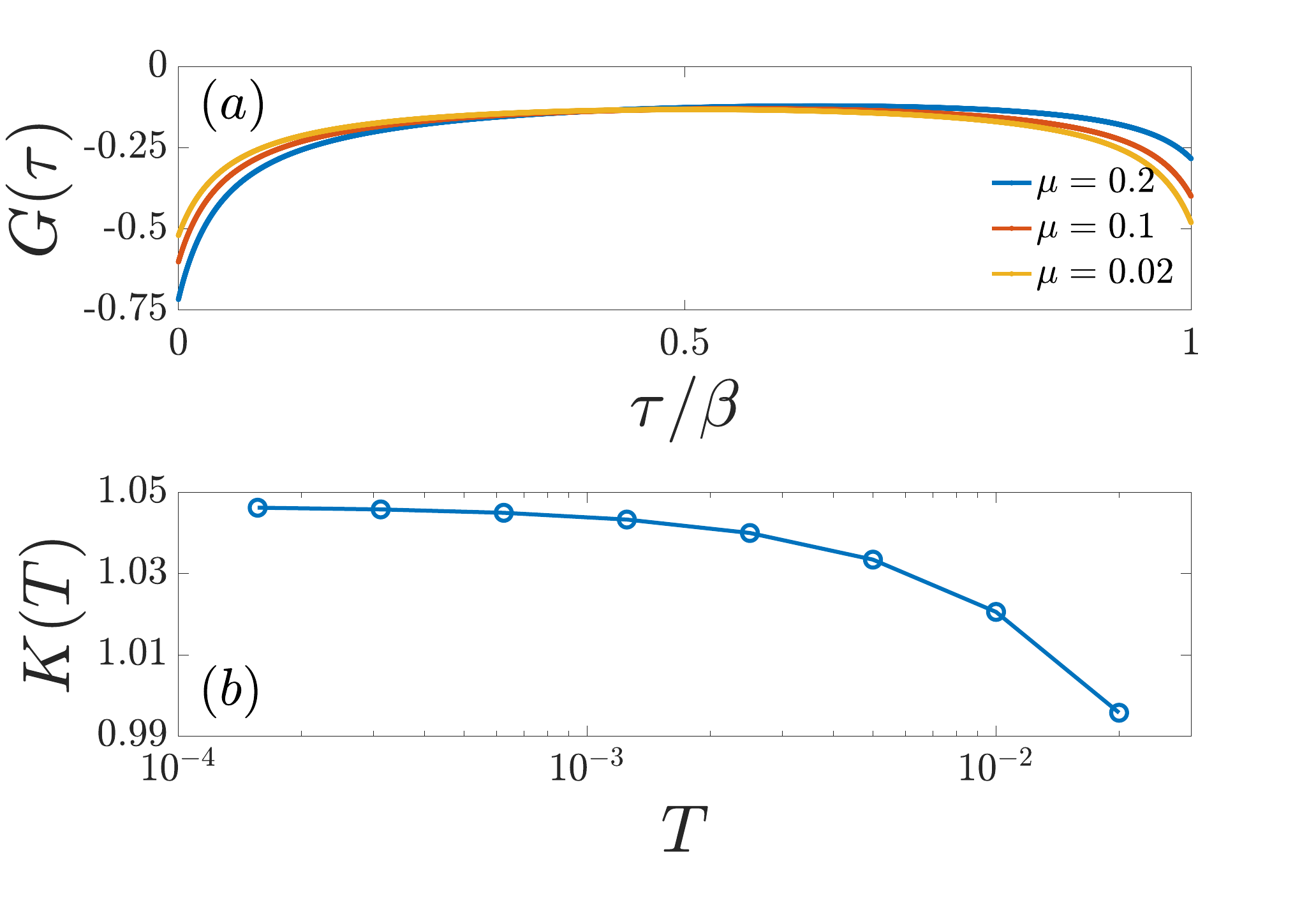}
  \caption{(a) Solution $G(\tau)$ of the SYK equation with $J = 1$,
  $\beta = 50$, and three values of $\mu$. (b) Compressibility $K(T)$ at
  low temperature.}
\label{fig:ksyk}
\end{figure}

Our goal is to calculate $K(0) = \lim_{T \to 0^+} K(T)$.
$G(\beta,\mu)$ is shown for $\beta = 50$ and $\mu = 0.2, 0.1, 0.02$ in Figure \ref{fig:ksyk}a.
As expected, $Q(\beta,\mu)$ is positive for $\mu > 0$ and decreases to
zero as $\mu \to 0$. 

In order to calculate $K(T)$ for each fixed $T$, we could simply compute $Q(\beta,\mu)/\mu$ for a small
value of $\mu$. However, this strategy suffers from rounding error due to catastrophic 
cancellation. To obtain a better approximation of
$K(T)$, we compute $Q$ by solving the SYK equation 
for $\mu = \mu_0/2^j$, with $j=1,\ldots,n$, and some choice of $\mu_0$
and $n$. We then use
Richardson extrapolation on the resulting values of $Q/\mu$ to obtain
the limiting value $K(T)$; see Ref. \onlinecite{dahlquist08} (Sec.
3.4.6) for a description of Richardson extrapolation.

We note that some care must be taken in the nonlinear iteration to avoid convergence to a
spurious exponentially-decaying solution. An effective strategy is to
compute the solution for a sequence of values of $\mu$: $\mu = j
\mu^*/n$, $j = 0,\ldots,n$, where $\mu^*$ is the desired value, and $n$
is chosen sufficiently large. For $\mu = 0$, we use the initial guess
$G(\tau) = -1/2$ in the nonlinear iteration, as above. For $\mu = j \mu^*/n$ with
$j>0$, we use the solution for $\mu = (j-1) \mu^*/n$ as an initial
guess. In many cases, taking $n = 1$ is sufficient.

We carry out this procedure 
for $\beta = 50,100,200,\ldots,6400$ with $\epsilon = 10^{-14}$,
$\epsilon_{\textrm{fp}} = 10^{-12}$, and $w$ taken sufficiently small to ensure convergence
of the nonlinear iteration. We take $\Lambda = 10\beta$, and have
verified that all calculations are converged with respect to this
parameter. The computations involve linear systems of at most $121$
degrees of freedom.

The computed values of $K(T)$ are shown in Figure
\ref{fig:ksyk}b.
From these values, we use Richardson extrapolation to estimate
$K(0)$:
\[K(0) \approx 1.0466998.\]

\section{Conclusion}\label{sec:Conclusion}

We have presented an efficient discrete Lehmann representation of imaginary time 
Green's functions based on the interpolative decomposition. In the low
temperature regime, it
requires far fewer degrees of freedom than standard discretizations,
and a similar number
to the recently introduced intermediate representation. The DLR basis
functions are explicit; they are exponentials, carefully chosen to ensure
stable and accurate approximation. This feature simplifies standard operations.
We have introduced algorithms which use standard numerical linear
algebra tools to efficiently
build the DLR basis and corresponding imaginary time and Matsubara
frequency grids. These algorithms also carry over to the intermediate
representation method.
We have demonstrated the DLR by solving the SYK equation to high
precision at low temperatures, with calculations taking on the order of
seconds on a laptop. Fortran and Python implementations of the algorithms
described in this paper are available in the library
\texttt{libdlr}. \cite{libdlr,kaye21}

\acknowledgements 
We thank Hugo Strand, Nikolay Prokof'ev, Boris Svistunov, Manas Rachh, Jeremy Hoskins, and Richard Slevinsky for
helpful discussions. The Flatiron Institute is a division of the
Simons Foundation.

\appendix

\section{DLR for bosonic Green's functions} \label{app:bosonic}

In this section, we argue that the DLR, derived using the
fermionic kernel $K(\tau,\omega)$ given by \eqref{eq:defK}, can also be applied directly
to bosonic Green's functions.

The truncated Lehmann representation for a bosonic Green's function is
given by
\begin{equation} \label{eq:tlehmannb}
  \gb(\tau)=-\int_{-\Lambda}^{\Lambda} \kb(\tau, \omega) \rhob(\omega) \, d\omega,
\end{equation}
where the bosonic kernel is given in nondimensionalized variables by
\begin{equation}
  \label{eq:bkernel}
  \kb(\tau, \omega)=\frac{e^{-\omega\tau}}{1-e^{-\omega}}.
\end{equation}
Although $\kb$ is singular at $\omega = 0$, for systems in which the $U(1)$
symmetry $\left< \hat{a} \right>=0$ and $\left< \hat{a}^\dag \right>=0$
(for $\hat{a}^\dag$/$\hat{a}$ the creation/anniliation operators)
is not spontaneously broken,
the singularity will be exactly cancelled out by a spectral density
vanishing to the appropriate order as $\omega \to 0$. Indeed, in this
case, the physical spectral density of a bosonic system has an explicit expression,
\begin{equation} \label{eq:rhoB}
  \begin{multlined}
  \rhob(\omega) = (1-e^{-\omega})\frac{2\pi}{Z} \\ \times \sum_{m,n} \left|\left<n|\hat{a}^\dag|m\right>\right|^2 e^{-E_m} \delta(E_n-E_m-\omega),
  \end{multlined}
\end{equation}
where $\left|n\right>$ and $\left|m\right>$ are eigenstates of the
many-body Hamiltonian with energies $E_n$ and $E_m$, respectively, and
$Z=\sum_n e^{-E_n}$ is the partition sum.

To handle this case, we simply rewrite \eqref{eq:tlehmannb} as
\[\gb(\tau)=-\int_{-\Lambda}^{\Lambda} K(\tau, \omega)
  \wt{\rhob}(\omega) \, d\omega,\]
where $K(\tau,\omega)$ is the fermionic kernel, and
\begin{equation} \label{eq:rho_reg}
  \wt{\rhob}(\omega) = \frac{1+e^{-\omega}}{1-e^{-\omega}} \rhob(\omega).
\end{equation}
The singularity in the factor $\frac{1+e^{-\omega}}{1-e^{-\omega}}$ is cancelled by
the factor $1-e^{-\omega}$ in $\rhob$ from \eqref{eq:rhoB};
otherwise, it is smooth and well-behaved.
Thus $\wt{\rhob}$ is integrable, and $\gb$ has the same Lehmann representation
as a fermionic Green's function, but with a modified spectral density.
The DLR method developed for fermionic Green's functions can therefore
be applied without modification.

\section{Proof of Theorem \ref{thm:dlr}} \label{app:thm1}

\begin{proofw}
The theorem follows from \eqref{eq:dlrvalid}, once we give a bound on the error
term. We have
\begin{align*}
  \abs{e(\tau)} &= \abs{\int_{-\Lambda}^\Lambda \errmat(\tau,\omega)
  \rho(\omega) \, d\omega} \\ &= \abs{\sum_{i=1}^M \wb{\ell}_i(\tau) \sum_{j=1}^N
  \errmat_{ij} \int_{-\Lambda}^\Lambda \wb{\ell}_j(\omega)
  \rho(\omega) \, d\omega} \\
  &\leq \norm{\errmat}_2 \sqrt{\sum_{i=1}^M \wb{\ell}_i^2(\tau)} \sqrt{\sum_{j=1}^N
  \paren{\int_{-\Lambda}^\Lambda \wb{\ell}_j(\omega)
  \rho(\omega) \, d\omega}^2}
\end{align*}
from the Cauchy-Schwarz inequality.

  From the definition of $\wb{\ell}_i$, we have $\norm{\sum_{i=1}^M
  \wb{\ell}_i^2(\tau)}_\infty = \norm{\sum_{k=1}^p
\ell_k^2(x)}_\infty$, where $\ell_k(x)$ are the Lagrange
polynomials at $p$ Chebyshev nodes on $[-1,1]$.
It follows from Lemma \ref{lem:sumsq}, proven in Appendix
\ref{app:sumsq}, that
\[\sum_{k=1}^p \ell_k^2(x) \leq 2.\]
For the last factor, we have
  \[\sum_{j=1}^N \paren{\int_{-\Lambda}^\Lambda \wb{\ell}_j(\omega)
  \rho(\omega) \, d\omega}^2 \leq \norm{\rho}_1^2 \sum_{j=1}^p \paren{\max_{x \in [-1,1]}
  \abs{\ell_j(x)}}^2.\]

Combining these results, we find
\[\norm{\errfun}_\infty \leq \sqrt{2 \sum_{j=1}^p \paren{\max_{x \in [-1,1]}
  \abs{\ell_j(x)}}^2} \norm{\errmat}_2 \norm{\rho}_1 =
  c \epsilon \norm{\rho}_1.\]
We note that $\sum_{j=1}^p \paren{\max_{x \in [-1,1]}
  \abs{\ell_j(x)}}^2$ depends only on $p$. Numerically, we find that it
  is approximately equal to $p$ for typical choices of $p$, implying
  $c \approx \sqrt{2p}.$
\end{proofw}

\section{Proof of Lemma \ref{lem:dlrstability}} \label{app:thm2}

\begin{proofw}
From \eqref{eq:krecover}, we have
  \[\gdlr(\tau) = \sum_{l=1}^r \wh{g}_l \sum_{k=1}^r \gamma_k(\tau)
  \kmat_{kl} = \sum_{k=1}^r \gamma_k(\tau) g_k\]
and similarly for $\hdlr$, so
\[\gdlr(\tau)-\hdlr(\tau) = \sum_{k=1}^r \gamma_k(\tau)
\paren{g_k-h_k}\]
and
  \[\abs{\gdlr(\tau)-\hdlr(\tau)} \leq \sqrt{\sum_{k=1}^r
  \gamma_k^2(\tau)} \, \norm{g-h}_2.\]
  Since $\gamma_k(\tau) = \sum_{i=1}^M \wb{\ell}_i(\tau) R_{ik}$, we have
\[\sum_{k=1}^r \gamma_k^2(\tau) \leq \norm{R}_2^2 \sum_{i=1}^M
  \wb{\ell}_i^2(\tau) \leq 2 \norm{R}_2^2.\]
  Here we have used Lemma \ref{lem:sumsq} from Appendix \ref{app:sumsq}, as in Appendix \ref{app:thm1}. The result follows from these
estimates.
\end{proofw}

\section{Bound on the sum of squares of Lagrange polynomials for Chebyshev nodes}
\label{app:sumsq}

The following lemma is used in Appendices \ref{app:thm1} and
\ref{app:thm2}:

\begin{lemma} \label{lem:sumsq}
  Let $\{\ell_k(x)\}_{k=1}^p$ be the Lagrange polynomials for the $p$
  Chebyshev nodes of the first kind on $[-1,1]$. Then
\[\sum_{k=1}^p \ell_k^2(x) \leq \sum_{k=1}^p \ell_k^2(1) = 2-1/p.\]
\end{lemma}
\begin{proof}
  The result follows from the identity
\begin{equation} \label{eq:lksqid}
  \sum_{k=1}^p \ell_k^2(x) = 1 + \frac{1}{2p} \paren{U_{2p-2}(x)-1},
\end{equation}
for $U_n(x)$ the degree $n$ Chebyshev polynomial of the second kind.
  Indeed, Ref. \onlinecite{nistdlmf} (Eqn. 18.14.1) gives that
\[\abs{U_n(x)} \leq U_n(1) = n+1,\]
and the desired result follows from this and \eqref{eq:lksqid}.

To prove \eqref{eq:lksqid}, we note that both the left and right
hand sides are polynomials of degree $2p-2$, so it suffices to show that
they agree in value and derivative at the $p$ Chebyshev nodes,
\[x_j = \cos\paren{\frac{2j-1}{2p}\pi},\]
for $j=1,\ldots,p$.

For the equality of values, the sine difference formula gives
\[U_{2p-2}(x_j) = \frac{\sin\paren{(2p-1)
\frac{2j-1}{2p}\pi}}{\sin\paren{\frac{2j-1}{2p}\pi}} = 1.\]
Since $\sum_{k=1}^p \ell_k^2(x_j) = 1$, this gives the equality.

For the equality of derivatives, we must show that
\[2\sum_{k=1}^p \ell_k(x) \ell_k'(x) = \frac{1}{2p} U'_{2p-2}(x)\]
for each $x = x_j$. Throughout the argument, we will use the formulas
for the derivatives of the Chebyshev polynomials of the first and second
kind, given by
\[T_n'(x) = n U_{n-1}(x)\]
and
\[U_n'(x) = \frac{(n+1) T_{n+1}(x) - x U_n(x)}{x^2-1}.\]
The cosine difference formula gives 
\begin{align*}
  \frac{U'_{2p-2}(x_j)}{2p}  &= \frac{(2p-1) T_{2p-1}(x_j) - x_j
  U_{2p-2}(x_j)}{2p \, (x_j^2-1)} \\
  &= \frac{(2p-1) \cos\paren{(2p-1) \frac{2j-1}{2p} \pi} - x_j}{2p
  \, (x_j^2-1)}
  \\
  &= \frac{x_j}{1-x_j^2}
\end{align*}
for the right hand side. For the left hand side, we have
  \[2 \sum_{k=1}^p \ell_k(x_j) \ell_k'(x_j) = 2 \ell_j'(x_j)
  = \sum_{\substack{k=0 \\ k \neq j}}^p \frac{2}{x_j-x_k}.\]
Our objective is therefore to show that
\begin{equation} \label{eq:diffeq}
\sum_{\substack{k=0 \\ k \neq j}}^p \frac{2}{x_j-x_k} = \frac{x_j}{1-x_j^2}
\end{equation}
for $j=1,\ldots,p$.

Define 
  \[f_j(x) = \sum_{\substack{k=0 \\ k \neq j}}^p \frac{1}{x-x_k},\]
  so that the left
hand side of \eqref{eq:diffeq} is equal to $f_j(x_j)$.
Let $\ell(x) =
\prod_{k=1}^p (x-x_k)$ be the node polynomial for the
Chebyshev nodes $x_j$. Then we have
\[f_j(x) = \frac{d}{dx} \log\abs{\ell(x)/(x-x_j)}.\]
We also have $\ell(x) = T_p(x)/2^{p-1}$, since $\ell(x)$ is a monic polynomial
of degree $p$ with zeros at the Chebyshev nodes. Therefore
\begin{align*}
  f_j(x) &= \frac{d}{dx} \log\abs{T_p(x)/(x-x_j)} \\
  &= \frac{(x-x_j) T_p'(x) - T_p(x)}{(x-x_j) T_p(x)} \\
  &= \frac{p (x-x_j) U_{p-1}(x) - T_p(x)}{(x-x_j) T_p(x)}
\end{align*}
and, using l'H{\^o}pital's rule, we find
\begin{align*}
  \sum_{\substack{k=0 \\ k \neq j}}^p \frac{1}{x_j-x_k} &= \lim_{x \to x_j} f_j(x) \\
  &= \lim_{x \to x_j} \frac{p (x-x_j) U_{p-1}(x) - T_p(x)}{(x-x_j)
  T_p(x)} \\
  &= \lim_{x \to x_j} \frac{p (x-x_j) U_{p-1}'(x)}{T_p(x) + p (x-x_j)
  U_{p-1}(x)} \\
  &= \lim_{x \to x_j} \frac{p \paren{p T_p(x) - x U_{p-1}(x)}/(x^2-1)}
  {T_p(x)/(x-x_j) + p U_{p-1}(x)} \\
  &= \frac{x_j}{2(1-x_j^2)}
\end{align*}
as was claimed.
\end{proof}

\bibliographystyle{ieeetr}
\bibliography{ref}
\end{document}